\documentclass[11pt]{amsart}

\usepackage[margin=1in]{geometry}
\usepackage{mathptmx,graphicx,amsmath,amsfonts,amsthm,amssymb,color,tabulary,caption,subcaption,todonotes,mathtools}
\usepackage{tikz-cd}

\newtheorem{thm}{Theorem}[section]
\newtheorem{lem}[thm]{Lemma}
\newtheorem{prop}[thm]{Proposition}

\newtheorem{cor}[thm]{Corollary}
\newtheorem{conj}[thm]{Conjecture}
\newtheorem{defn}[thm]{Definition}

\newtheorem{ques}[thm]{Question}

\theoremstyle{definition}

\DeclareMathOperator{\Id}{e}

\DeclareMathOperator{\Aut}{Aut}
\DeclareMathOperator{\Out}{Out}
\DeclareMathOperator{\Inn}{Inn}
\DeclareMathOperator{\Homeo}{Homeo}
\DeclareMathOperator{\Homeop}{Homeo^+}

\DeclareMathOperator{\Hom}{Hom}

\DeclareMathOperator{\CO}{CO}
\DeclareMathOperator{\COmin}{CO_{min}}
\DeclareMathOperator{\COfin}{CO_{fin}}
\DeclareMathOperator{\COfintran}{CO_{fin,tran}}
\DeclareMathOperator{\COb}{CO_{blow}}
\DeclareMathOperator{\CObtran}{CO_{blow,tran}}
\DeclareMathOperator{\CObtranrot}{CO_{blow,tran,rot}}
\DeclareMathOperator{\COrot}{CO_{rot}}
\DeclareMathOperator{\LO}{LO}
\DeclareMathOperator{\LOtran}{LO_{tran}}
\DeclareMathOperator{\LOflag}{LO_{flag}}
\DeclareMathOperator{\BCO}{BCO}

\DeclareMathOperator{\Stab}{Stab}

\DeclareMathOperator{\MCG}{MCG}
\DeclareMathOperator{\rank}{rank}
\DeclareMathOperator{\im}{Im}

\newcommand{\ZZ}{\mathbb{Z}}      
\newcommand{\QQ}{\mathbb{Q}}   
\newcommand{\RR}{\mathbb{R}}

\newcommand{\TT}{\mathbb{T}}

\newcommand{\into}{\hookrightarrow}
\newcommand{\onto}{\twoheadrightarrow}
\newcommand{\M}{\{-1,0,1\}^{G^3}}
\newcommand{\ie}{\emph{i.e.\ }}
\newcommand{\calo}{\mathcal{O}}

\begin{document}

\title{Spaces of invariant circular orders of groups}

\author{Hyungryul Baik}
\address{Mathematisches Institut, 
Rheinische Friedrich-Wilhelms-Universit\"{a}t Bonn, 
Endenicher Allee 60,
53115 Bonn, Germany.}
\email{baik@math.uni-bonn.de}

\author{Eric Samperton}
\address{Department of Mathematics, UC Davis, 1 Shields Avenue, Davis, CA, 95616, USA.}
\email{egsamp@math.ucdavis.edu}

\begin{abstract}
Motivated by well known results in low-dimensional topology, we introduce and study a topology on the set $\CO(G)$ of all left-invariant circular orders on a fixed countable and discrete group $G$.  $\CO(G)$ contains as a closed subspace $\LO(G)$, the space of all left-invariant \emph{linear} orders of $G$, as first topologized by Sikora.  We use the compactness of these spaces to show the sets of non-linearly and non-circularly orderable finitely presented groups are recursively enumerable.  We describe the action of $\Aut(G)$ on $\CO(G)$ and relate it to results of Koberda regarding the action on $\LO(G)$.  We then study two families of circularly orderable groups: finitely generated abelian groups, and free products of circularly orderable groups.  For finitely generated abelian groups $A$, we use a classification of elements of $\CO(A)$ to describe the homeomorphism type of the space $\CO(A)$, and to show that $\Aut(A)$ acts faithfully on the subspace of circular orders which are not linear.  We define and characterize Archimedean circular orders, in analogy with linear Archimedean orders.  We describe explicit examples of circular orders on free products of circularly orderable groups, and prove a result about the abundance of orders on free products.  Whenever possible, we prove and interpret our results from a dynamical perspective.


\smallskip
\noindent \textbf{MSC classes:} 20F60, 37E10, 20F10, 20F65.
\end{abstract}

\maketitle

\section{Introduction}\label{sec:intro}
Group actions on the circle and circular orders are closely related by a well-known fact: a countable group $G$ acts faithfully on $S^1$ by orientation-preserving homeomorphisms if and only if $G$ admits a left-invariant circular order.  This fact, which we restate and prove in Proposition \ref{prop:faithfulactioniffCO} of \S \ref{ss:threeperspectives}, connects topological dynamics on $S^1$ to an algebraic property of groups called \emph{circular orderability}. Roughly speaking, a circular order on a group $G$ is a way to consistently assign either a clockwise or counterclockwise orientation to each ordered 3-tuple of elements of $G$, in such a way that orientations are invariant under left multiplication.

In general, deciding linear orderability or circular orderability of finitely-presented groups is a very difficult question.  In fact, it is provably impossible to decide, as these are both Markov properties. A property P of finitely presented groups is called \emph{Markov} if there exists a group without property P which can not be embedded in any group with P. Linear orderability is Markov, for example, because $\Homeop(\RR)$ is torsion free, and every linearly orderable group can be realized as a subgroup of $\Homeop(\RR)$.  Circular orderability is Markov because circular orders restrict to subgroups and the only finite groups which are circularly orderable are cyclic, see Proposition \ref{lem:torsionandCO}.  Markov properties are always undecidable \cite{Rabin58}.  However, determining \emph{non-orderability} (circular or linear) is a decision problem in the complexity class $\mathsf{RE}$; in other words, it is in principle possible to list all finitely-presented groups that are not orderable.  See Corollaries \ref{cor:CORE} and \ref{cor:LORE}, which prove these facts using a compactness argument.

Despite the impossibility of deciding circular orderability for arbitrary groups, it is a phenomenon that occurs frequently in low-dimensional topology.  For example, Fuchsian groups, braid and mapping class groups, and fundamental groups of 3-manifolds which support a taut foliation, are all circularly orderable \cite{Dehornoy94,CalegariDunfield03}.  However, it is not known if (circular) orderability is a decidable property for the class of 3-manifold groups.

By studying orderability abstractly, we hope, in further work, to better understand the geometry of 3-manifolds.  For example, the first author has elsewhere shown that Fuchsian groups can be characterized as those groups which admit faithful, orientation-preserving topological actions on $S^1$ with three invariant laminations with additional properties \cite{Baik15}.  It is conjectured that there should be a similar characterization of Kleinian groups (see \cite{AlonsoBaikSamperton15} for a precise statement, and \cite{Baik15} for more background).  Furthermore, \emph{linear} orderability has recently been the source of attention from Heegaard-Floer theorists due to the conjecture of \cite{BoyerGordonWatson13}, which contends that an irreducible rational homology 3-sphere is an L-space if and only if its fundamental group is \emph{not} linearly orderable. 

Finally, we remark that Sikora first introduced a topology on the space $\LO(G)$ of left-invariant linear orders of a countable group \cite{Sikora04}.  He used this topology to give a new proof---via a compactness argument---of the existence of universal Gr\"obner bases for ideals in polynomial rings.  The topology of $\LO(G)$ was further exploited to show that $\LO(G)$ is either finite or uncountable---in which case, it contains a Cantor set \cite{Navas10,Linnell11}.  An essential concept in both \cite{Navas10} and \cite{Linnell11} is that of a Conradian order.  We do not know of an appropriate analog of Conradian orders for circular orders.

The organization of this paper is as follows.  In  Section \ref{sec:setup}, we define ``circular order" and the topology on the set $\CO(G)$ of all left-invariant circular orders of $G$, and study basic properties and various subspaces of $\CO(G)$.  We also introduce the action of $\Aut(G)$ on $\CO(G)$.  A result of Koberda \cite{Koberda11} implies this action is faithful when $G$ is residually torsion-free nilpotent, but this raises the question of when $\Aut(G)$ acts faithfully on its collection of genuine circular orders---\ie circular orders that are not linear.  We show in Corollary \ref{cor:abeliangenuinefaithful} of \S\ref{ss:abeliantopology} that the action is faithful for $G$ a circularly-orderable, finitely-generated abelian group.

In Section \ref{sec:abeliangroups}, we determine $\CO(A)$ where $A$ is a finitely generated abelian group.  In \S\ref{ss:abelianclassification}, we use known results to classify circular orders on $A$.  In \S\ref{ss:abeliantopology}, we use this classification to construct convenient dense subsets of $\CO(A)$, which allows us to describe the topology on $\CO(A)$.  The proof that these subsets are dense uses Noetherian induction on the poset of finitely generated abelian groups with cyclic torsion subgroups, wherein two such groups are comparable if one surjects onto the other.  It is straightforward to see that if $A$ is finite cyclic, $\CO(A)$ is finite and discrete.  If $A$ does not have a cyclic torsion subgroup, then Lemma \ref{lem:torsionandCO} implies $\CO(A)$ is empty.  Finally, if $A$ has positive rank and a cyclic torsion subgroup, then the main result of the \S\ref{ss:abeliantopology}, Theorem \ref{thm:cantorset}, shows $\CO(A)$ is a Cantor set.  In \S\ref{ss:archimedean}, we define and characterize Archimedean circular orders.  We include this in Section \ref{sec:abeliangroups} because our characterization and a theorem of H\"{o}lder implies such orders are always abelian.

In Section \ref{sec:freeproducts}, we study free products of circularly orderable groups.  In \S\ref{ss:freeexistence}, we give two proofs of the existence of circular orders on a free product of circularly orderable groups.  One is straightforward but nonconstructive, using the Kurosh subgroup theorem; the other is an explicit and unique construction satisfying a kind of lexicographical condition.  See Theorems \ref{thm:free} and \ref{thm:freeproducts}, respectively.  In \S\ref{ss:freeabundance}, we discuss and analyze the abundance of orders on free products, following Rivas \cite{Rivas12}.  By adapting the argument Rivas uses to show that the space of linear orders on a free product of linearly orderable groups is a Cantor set, we show in Theorem \ref{thm:COminperfect} that for infinite circularly orderable groups $G$ and $H$, the set of circular orders on $G*H$ for which $G$ acts minimally has no isolated points in $\CO(G*H)$.  We conjecture that $\CO(G*H)$ is always a Cantor set if $G$ and $H$ are nontrivial circularly orderable groups, but, for example, it is unclear if the lexicographical orders of Theorem \ref{thm:freeproducts} can be isolated or not.

\subsection{Acknowledgments} 
We thank Dawid Kielak and Sanghyun Kim for helpful discussions especially regarding Section 4.2.  The second author thanks Universit\"{a}t Bonn for hosting him, during which time part of this work was carried out.  The first author was partially supported by the ERC Grant Nb. 10160104. 

%%%%%
%%%%%
\section{Topology on the spaces of circular orders of groups}\label{sec:setup}
\subsection{Three perspectives on circular orders}\label{ss:threeperspectives}
The following is a standard definition \cite{Calegari04} of a circular order of a set $G$. 

\begin{defn}
Let $G$ be a set with at least four elements. A circular order on $G$ is a choice of total order on $G \setminus \{p\}$ for every $p \in G$, such that if $<_p$ is the total order defined by $p$, and $p,q \in G$ are two distinct elements, the total orders $<_p, <_q$ differ by a cut on their common domain of definition. That is, for any $x, y$ distinct from $p, q$ , the order of $x$ and $y$ with respect to $<_p$ and $<_q$ is the same unless
$$x <_p q <_p y \text{ or } y <_p q <_p x,$$
in which case we have
$$y <_q p <_q x \text{ or } x <_q p <_q y,$$
respectively.  We also  say that the order $<_q$ on $G\setminus \{p,q\}$ is obtained from the order  $<_p$ on $S \setminus \{p\}$ by cutting at $q$.
\end{defn}

For example, every linear order $<$ on a set $G$ gives rise to a circular order.  For each $p \in G$, we define the cut $<_p$ by
$$<_p(x,y) = \begin{cases} <(y, x) & \text{ if } x < p < y,\ p < y < x \text{ or } y < x < p \\ 
<(x,y) & \text{ if } y < p < x,\ p < x < y \text{ or } x < y < p \end{cases}$$
Here we are conflating $<$ with the characteristic function $<: G^2 \to \{0,+1\}$ of the \emph{positive cone}: all pairs $(x,y)$ such that $x<y$.  Likewise for $<_p$.  It is easy to see that this satisfies the definition above.  Note that there is another identification of $<$ with the opposite circular order $<_p'(x,y):= <_p(y,x)$.  If the reader draws pictures, she will see why we chose $<_p$ as the preferred identification: the natural way to compactify a line to a circle is to make the orientations agree.

Now let $G$ be a countable discrete group.\footnote{Countability of $G$ is not a very strong restriction, cf.\ \cite{Calegari04}.} Then one can consider circular orders on $G$ as a set. However, since we have a group structure, the natural orders to consider are those preserved by the action of $G$ on itself by left multiplication, so that $x<_p y$ if and only if $gx <_{gp} gy$ for all $g$. We find it is hard to get our hands on this definition.  Fortunately, there are alternative ways to think about circular orders on groups. We will consider certain homogeneous cocycles on $G^3$. Define the degenerate set $\Delta_n(G)$ of $G^n$ (the product of $n$-copies of $G$) to be the set $\Delta_n(G) = \{ (g_1, \ldots, g_n) \in G^n \mid g_i = g_j \mbox{ for some } i \neq j \}$ of all $n$-tuples with some repeated elements.

\begin{defn}[Alternative definition]\label{defn:CO}
A \emph{circular order} on a set $G$ is a map $\phi : G^3 \to \{-1, 0, 1\}$ with the following properties:
\begin{itemize}
\item[(DV)] $\phi$ kills precisely the degenerate set, \ie
$$\phi^{-1}(0) = \Delta_3(G).$$

\item[(C)] $\phi$ is a 2-cocycle, \ie
$$\phi(g_1, g_2, g_3) - \phi(g_0, g_2, g_3)
+ \phi(g_0, g_1, g_3) - \phi(g_0, g_1, g_2) = 0$$
for all $g_1, g_2, g_3, g_4 \in G$. 
\end{itemize}
If $\phi(x,y,z) = +1$, we say $(x,y,z)$ is a \emph{positively oriented} triple.  Likewise, if $\phi(x,y,z) = -1$, we say $(x,y,z)$ is a \emph{negatively oriented} triple. If $\phi(x,y,z) = 0$, we say $(x,y,z)$ is a \emph{degenerate} triple.  Furthermore, if $G$ is a group, then a \emph{left-invariant circular order} on $G$ is a circular order on $G$ as set that also satisfies the homogeneity property:
\begin{itemize}
\item[(H)] $\phi$ is homogeneous, \ie
$$\phi(g_0, g_1, g_2) = \phi(gg_0, gg_1, gg_3)$$
for all $g\in G$ and $(g_0, g_1, g_2) \in G^3$.
\end{itemize}
\end{defn}

By abuse of language, we will often refer to a ``left-invariant circular order of a group" simply as a ``circular order."

\begin{lem}[Construction 2.3.4 of \cite{Calegari04}]\label{lem:definitionsagree}
The definition of a circular order as a homogeneous cocycle coincides with the definition of a circular order in terms of cuts.
\end{lem} 

\begin{proof}
Given a cocycle $\phi$ satisfying (DV), let $x<^\phi_p y$ if and only if $\phi(y,p,x) = +1$.

Conversely, given $\{<_p\}_{p\in G}$, define
$$\phi(x,y,z)=\begin{cases} +1 & \text{ if } z<_y x \\
-1 & \text{ if } x<_y z \\
0 & \text{ otherwise.}\end{cases}$$
\end{proof}

We observe two additional properties about circular orders:

\begin{itemize}
\item[(IC)] \emph{$\phi$ is invariant under cyclic permutations, \ie
$$\phi(g_0, g_1, g_2) = \phi(g_1, g_2, g_0)$$
for all $(g_0, g_1, g_2) \in G^3$.}

\item[(AT)] \emph{$\phi$ is antisymmetric with respect to transposing two arguments, e.g.\
$$\phi(g_0, g_1, g_2) = -\phi(g_0, g_2, g_1)$$
for all $(g_0, g_1, g_2) \in G^3$.}
\end{itemize}

\noindent To deduce (IC), let $g_3 = g_0$ in (C), so that (DV) implies (IC).  Property (AT) follows similarly by letting $g_3 = g_1$. Properties (IC) and (AT) are the properties that one intuitively expects to hold for circular orders, while (C) is roughly a compatibility condition.

Note that the identification of a linear order with a circular order given above can be converted to the cocycle picture by sending a linear order $<$ to the cocycle $c_<$ defined by
$$c_<(x,y,z) := \begin{cases}
+1 & \text{if } x < y < z,\ y < z < x, \text{ or } z < x < y \\
-1 & \text{if } x < z < y,\ y < x < z, \text{ or } y < x < z\\
0 & \text{otherwise.}
\end{cases}$$

Our subsection title promised three perspectives on circular orders.  The third has already been mentioned in the introduction.  We restate and sketch the proof of this well known result here, in order to introduce notation and terminology, as well as the basic idea used to establish later results.

\begin{prop}[Theorem 2.2.14 of \cite{Calegari04}]
\label{prop:faithfulactioniffCO}
A countable group $G$ acts faithfully on $S^1$ by orientation-preserving homeomorphisms if and only if $G$ admits a left-invariant circular order.
\end{prop}

\begin{proof}
We will consider $S^1$ as the unit interval with $0 = 1$, equipped with the standard positive circular order as a set.

Let $c$ be a circular order on $G$.  Enumerate $G$. We first construct a set map $i: G\into S^1$.  Send $g_0$ to $0$ and $g_1$ to $\frac{1}{2}$.  Send $g_2$ to wherever $c$ tells you to send $g_2$ and continue.  More precisely, let $i(g_2) = \frac{3}{4}$ if $c(g_0,g_1,g_2) = +1$; otherwise, let $i(g_2) = \frac{1}{4}.$ Continue with this construction in the obvious way to define $i$ on all of $G$.

Since $G$ acts on itself bijectively by left multiplication, we get an action of $G$ on $i(G)$.  By the way we used the circular order to build $i$, this action preserves the order information of $i(G)$.  Now extend this action to an action $r_c:G\into \Homeop(S^1)$ on all of $S^1$ by, for example, making $r_c(g)$ act linearly on the interval gaps between adjacent elements in $i(G)$.  We call $r_c$ the \emph{dynamical realization} of $c$ (as an analogue of dynamical realization of left-invariant linear orders in \cite{Navas10}).  If $e$ is the identity element of $G$, observe that the point $i(e) \in S^1$ has trivial stabilizer.

For the converse direction, we give two constructions.  First, suppose there exists a point $p \in S^1$ that is trivially stabilized by the $G$ action $\phi$. Define
\begin{equation}\label{eqn:blergh}
c_\phi(x,y,z) := \begin{cases}
-1 & \text{ if } (\phi(x)(p),\phi(y)(p),\phi(z)(p)) \text{ is negatively oriented in } S^1,\\
+1 & \text{ if } (\phi(x)(p),\phi(y)(p),\phi(z)(p)) \text{ is positively oriented in } S^1,\\
0 & \text{otherwise.}\end{cases}
\end{equation}

For the second construction, we only need to use faithfullness of $\phi$ and the fact that $S^1$ is separable.  Enumerate some countable dense subset $Q=\{x_1,x_2,\dots\}$ of $S^1$.  Given a triple $(a,b,c)$ of distinct elements of $G$, let
$$m = m(a,b,c) := \min \{i \mid a \cdot x_i \neq b \cdot x_i \neq c \cdot x_i \neq a \cdot x_i \}.$$
The minimum exists because $Q$ is dense and $\phi$ is faithful.  Now set
$$c(a,b,c) := c_{S^1}(a \cdot x_m, b \cdot x_m,c \cdot x_m).$$
We leave it to the reader to check that $c$ is a circular order on $G$.
\end{proof}

We remark that there are completely analogous constructions for left-invariant \emph{linear} orders, which fact we record here as

\begin{prop}
\label{prop:faithfulactioniffLO}
A countable group $G$ admits a faithful action on $\RR$ by orientation-preserving homeomorphisms if and only if $G$ admits a left-invariant linear order. \hfill\qed
\end{prop}

This too appears in Theorem 2.2.14 of \cite{Calegari04}, although it is possible this result was known earlier.

These last two propositions imply a third way to construct a circular order from a linear order: given a linear order we can build an action on $\RR$.  Compactify $\RR$ to $S^1$ and extend the action so that $\infty$ is a fixed point. Proposition~\ref{prop:faithfulactioniffCO} now provides a circular order.  The reader can check that this construction of a circular order from a linear order is equivalent to the previous two constructions, the first involving cuts and the second cocycles.

In the proof of Proposition \ref{prop:faithfulactioniffCO}, we described the construction of an injective homomorphism $r_c: G \to \Homeop(S^1)$ called the \emph{dynamical realization} of a circular order $c$. $r_c$ is almost well-defined; the action of $G$ on $\overline{i(G)}$ is completely well-defined but the action on the interval gaps is well-defined up to semi-conjugacy. Since the stabilizer of each gap acts on the open interval freely, by H\"older's theorem (Theorem 6.10 of \cite{Ghys01}), this action is always semi-conjugate to a group of translations. 

Another remark on the dynamical realization is that $r_c$ is not just a group homomorphism from $G$ to $\Homeop(S^1)$, but rather a group homomorphism with a marked point $p$. In the construction, $p$ corresponds to $i(e)$ where $e$ is the identity element. One can recover $c$ from $r_c$ with $p$ by by declaring a distinct triple of elements $(g_1, g_2, g_3)$ to be positively oriented if and only if $(g_1p, g_2p, g_3p)$ is positively oriented with respect to the natural circular order on $S^1$. On the other hand, when $G$ acts on $S^1$ with two trivially stabilized points $p, q$, the circular order on the orbit of $p$ and one on the orbit of $q$ might give two different circular orders on $G$. The choice of $p$ matters less for bi-invariant circular orders. Here, a \emph{bi-invariant} circular order of $G$ is a left-invariant circular order that is also right-invariant.  In other words, the cocycle describing the order is also homogeneous on the right.

\begin{prop}\label{prop:well_defined_biinvariant_co}
Given a circular order $c$ of $G$, let $r_c$ be the dynamical realization with a marked point $p$. $c$ is bi-invariant if and only if one can choose any point in the orbit of $p$ under $r_c(G)$ as a new marked point. 
\end{prop}
\begin{proof}
This is straightforward. $c$ being bi-invariant means that $(ap, bp, cp)$ is oriented in the same way as $(adp, bdp, cdp)$ for any $a, b, c, d \in G$. But this is equivalent to saying that $dp$ can be used as a marked point to recover $c$ from $r_c$. 
\end{proof}

This means that the set of bi-invariant circular orders is essentially in one-to-one correspondence with the set of orbit classes of trivially stabilized points under its dynamical realizations.

%%%%%
\subsection{The space of circular orders}\label{ss:CO}
We now describe a topology on the set of all circular orders of a fixed group.

\begin{defn}
Given a group $G$, $\CO(G)$ denotes the set of all left-invariant circular orders on $G$.
\end{defn}

We defined a circular order as a map from $G^3$ into $\{-1, 0, 1\}$ satisfying some axioms. We topologize $\CO(G)$ as a subspace of the space $\{-1, 0, 1\}^{G^3}$ of all maps from $G^3$ to $\{-1, 0, 1\}$, where $\{-1, 0, 1\}^{G^3}$ is equipped with the Tychonoff topology induced from the discrete topology on $\{-1, 0, 1\}$. If $T \in G^3\setminus\Delta(G)$, we define
$$B_T:= \{c \in \CO(G) \mid c(T) = +1\}.$$
Similarly, given $T_1,\dots,T_n \in G^3\setminus\Delta(G)$, we let
$$B_{T_1,\dots,T_n} = \bigcap_{i=1}^n B_{T_i}.$$

\begin{lem}\label{lem:COsubbasis}
The collection $\{B_T\}_{T \in G^3\setminus\Delta(G)}$ forms a subbasis for the topology on $\CO(G)$.
\end{lem}

\begin{proof}
Recall that the set of triples $\{-1, 0, 1\}^{G^3}$ is a Cantor set with a subbasis given by sets of the form
$$B_{T,\epsilon} = \{f: G^3 \to \{-1,0,1\} \mid f(T) = \epsilon\}$$
where $T \in G^3$ and $\epsilon \in \{-1,0,1\}$. %Similar to before, given $(T_1,\epsilon_1),\dots,(T_n,\epsilon_n) \in G^3 \times \{-1,0,1\}$ we let
%$$B_{(T_1,\epsilon_1),\dots,(T_n,\epsilon_n)} = \bigcap_{i=1}^n B_{T_i,\epsilon_i}.$$
%Thus the $B_{(T_1,\epsilon_1),\dots,(T_n,\epsilon_n)}$ form a basis.
So a subbasis of $\CO(G)$ is given by sets of the form
$$B_{T,\epsilon} \cap \CO(G).$$
Note that if $\epsilon=+1$,
$$B_{T,+1} \cap \CO(G) = B_T.$$
So we should argue that we can throw out all sets of the form $B_{T,-1} \cap \CO(G)$ and $B_{T,0} \cap \CO(G)$, yet still have a subbasis.

Our subbasis does not need sets of the form $B_{T,0} \cap \CO(G)$ because
$$\CO(G) \cap B_{T,0} = \CO(G)$$
for all $T\in \Delta(G)$, and
$$\CO(G) \cap B_{T,0} = \emptyset$$
for all $T \in G^3 \setminus \Delta(G)$.

The subbasis does not needs sets of the form $B_{T,-1} \cap \CO(G)$ because
$$B_{T,-1} \cap \CO(G) = B_{\tau \cdot T, +1} \cap \CO(G)$$
where $\tau$ is any transposition.
\end{proof}

We first want to understand how $\CO(G)$ sits inside $\{-1, 0, 1\}^{G^3}$.  Since $\CO(G)$ is a subspace of the Cantor set $\{-1, 0, 1\}^{G^3}$, $\CO(G)$ is totally disconnected.  In fact, we have a little more.

\begin{prop}\label{prop:COclosed}
$\CO(G)$ is a closed subspace of $\{-1, 0, 1\}^{G^3}$.
\end{prop}

\begin{proof}
Observe that each condition in Definition \ref{defn:CO} is a closed condition. That is, if we write
$$\begin{aligned} \CO(G) = &\left(\bigcap_{T \in \Delta(G)} B_{T,0}\right)\cap\left( \bigcap_{T \notin \Delta(G)} B_{T,-1} \cup B_{T,+1}\right)\\
&\cap\left(\bigcap_{Q \in G^4} \{f \in \{-1, 0, 1\}^{G^3} \mid df(Q) = 0\}\right) \\
&\cap\left(\bigcap_{g \in G, T \in G^3} \{f \in \{-1, 0, 1\}^{G^3} \mid f(g\cdot T) = f(T) \}\right)\end{aligned}$$
then it is easy to check that any subset of $\{-1, 0, 1\}^{G^3}$ of one of the forms
$$B_{T,0},$$
$$B_{T,-1} \cup B_{T,+1},$$
$$\{f \mid df(Q) = 0\},$$
or
$$\{f \mid f(g\cdot T) = f(T) \}$$
is closed.
\end{proof}

%The set of triples $\{-1, 0, 1\}^{G^3}$ is a Cantor set whose topology we can construct as follows. If $(\phi_i)$ is a sequence in $\{-1, 0, 1\}^{G^3}$ converging to $\phi$ then, defining
%$$A_i := \{ (g_0, g_1, g_2) \in G^3 : \phi_i(g_0, g_1, g_2) = \phi(g_0, g_1, g_3)\} $$
%we have that for all $i$ there exists $S_i \subset A_i$ such that $\bigcup_i S_i = \{-1, 0, 1\}^{G^3}$ and the $S_i$ are increasing.  Let $\phi_i \in \CO(G) \subset \{-1, 0, 1\}^{G^3}$ and suppose $\phi_i \to \phi \in \{-1, 0, 1\}^{G^3}$. We check properties $(DV), (H)$ and $(C)$.

%Let $(g_0,g_1,g_2) \in \Delta_3(G)$. Since the $S_i$ filter $\{-1, 0, 1\}^{G^3}$ the triple $(g_0,g_1,g_2)$ is in $S_i$ for all $i$ large enough.  Going the other way, if we have that $\phi(g_0,g_1,g_2) = 0$ then there is some $n$ such that for all $i$ large enough $(g_0,g_1,g_2) \in S_i$.  Thus $\phi_i(g_0,g_1,g_2)=0$.  (In fact, $\phi_i(g_0,g_1,g_2)=0$ \emph{for all} $i$.)

%We show $\phi(g\cdot g_0,g\cdot g_1,g\cdot g_2) = \phi(g_0,g_1,g_2)$.  To see this, take the maximum of the two minimal indices $i$ and $i_g$ such that $(g_0,g_1,g_2) \in S_i$ and $(g\cdot g_0,g\cdot g_1,g\cdot g_2) \in S_{i_g}$. Call this max $j$.  Then the two triples are both in $S_k$ for $k$ bigger than $j$.  Since all of these $\phi_k$ satisfy are homogeneous with respect to $(g_0,g_1,g_2)$ and $g$, $\phi$ is too.

%Showing $\phi$ is a cocycle is analogous: just take the max of the four different indices coming from the four different triples in the cocycle equation.

\begin{cor}
\label{cor:CORE}
There is an algorithm that, when input a finitely presented group $G = \langle S \mid R\rangle$, outputs (in finite time) $\mathsf{NO}$ together with an obstruction certificate if $G$ is not circularly orderable; if $G$ is circularly orderable, the algorithm runs forever.  In particular, non-circular-orderability is a recursively enumerable property.
\end{cor}

\begin{proof}
We apply the finite intersection property.  The proof of the previous proposition shows that $\CO(G)$ is a closed subset of a Cantor set defined by the intersection of the collection
$$\mathcal{C} := \{B_{T,0}\}_{T \in G^3}\cup \{B_{T,-1} \cup B_{T,+1}\}_{T \in G^3}\cup\{\{f \mid df(Q) = 0\}\}_{Q \in G^4} \cup \{\{f \mid f(g\cdot T) = f(T) \}\}_{g \in G,T \in G^3}$$
of closed subsets each defined locally, \ie each involving only finitely many triples.  Since $\CO(G)$ is a compact subspace of $\{-1,0,+1\}^{G^3}$, if $\CO(G) = \emptyset$, the finite intersection property implies that the intersection of some finite collection of sets in $\mathcal{C}$ is empty.  Combining these two observations, if $G$ is not circularly orderable, some finite obstruction to circular orderability (in the form of a finite family of triples in $G^3$) must exist.

The brute-force algorithm proceeds by enumerating every finite subset of triples of elements of $G$ and checking if that finite subset can possibly satisfy the homogeneous 2-cocycle conditions.  If a subset is found that can not, the algorithm outputs NO, together with that subset.  If $G$ is not circularly orderable, the previous paragraph shows such a subset will always eventually be found. If $G$ is circularly orderable, such a finite subset will never be found, so the algorithm runs without terminating. Finally, to see that non-circular-orderability is a recursively enumerable property, simply enumerate all finitely presented groups while applying the brute-force search in parallel as you enumerate; whenever an obstruction is found for a group, add that group to the list of non-circularly-orderable groups.
\end{proof}

%%%%%
\subsection{Subspaces of $\CO(G)$}\label{ss:subspaces}
As discussed, the definition of left-invariant circular order can be strengthened to require that the order be homogeneous with respect to \emph{right} multiplication as well, and such an order is called a bi-invariant circular order.

\begin{defn}
$\BCO(G) \subset \CO(G)$ is the subspace of bi-invariant circular orders.
\end{defn}

For example, if $A$ is an abelian group, all left invariant orders are automatically bi-invariant, so $\BCO(A) = \CO(A)$.  More generally, arguing as we did in Proposition \ref{prop:COclosed}, it is easy to check

\begin{lem}\label{lem:BCOclose}
$\BCO(G)$ is a closed subspace of $\CO(G)$. \hfill\qed
\end{lem}

In Section 2.1 we constructed circular orders from linear orders.  Now we define a map from the set $\LO(G)$ of left-invariant linear orders to $\CO(G)$:
$$\begin{aligned} i: \LO(G) &\into \CO(G) \\
< &\mapsto c_<\end{aligned}$$
where $c_<$ was defined in equation \ref{eqn:blergh} of Subsection \ref{ss:threeperspectives}. The map $i$ is obviously injective.  Suppose $<^1, <^2$ are two different left-invariant linear orders on $G$.  Then there exists a non-trivial element $g$ of $G$ such that $g >^1 \Id $ and $g <^2 \Id$.  But then, $c_{<^1}(g^{-1}, \Id, g) = +1$ and $c_{<^2}(g^{-1}, \Id, g) = -1$.  Hence, $i(<^{1}) \neq i(<^{2})$. 

\begin{defn}
Let $\LO(G) := i(\LO(G))$ denote the subspace of $\CO(G)$ consisting of all left-invariant linear orders.  A circular order that can not be realized as arising from a linear order is called a \emph{genuine} or \emph{proper} circular order.  $\CO_g(G):= \CO(G) \setminus \LO(G)$ denotes the subspace of genuine circular orders.
\end{defn}

For example, since $\Homeop(\RR)$ is torsion-free, any circular order on a finite cyclic group is genuine. If $G=\ZZ/n\ZZ$ is a finite cyclic group, then there are precisely $\phi(n)$ circular orders on $G$, where $\phi$ is Euler's totient function.  Let $\Sigma_{g,n}$ be a surface of genus $g$ with $n$ punctures, and let $\MCG(\Sigma_{g,n})$ or $\MCG_{g,n}$ denote the mapping class group of the surface $\Sigma_{g,n}$. Since $\MCG_{g,n}$ has torsion elements, it is not left-orderable, hence $\CO(\MCG_{g,n}) = \CO_g(\MCG_{g,n})$.

Topologizing a space of orders is not a new idea.  Indeed, as discussed in the introduction, Sikora introduced a topology on $\LO(G)$ \cite{Sikora04} in which a subbasis consists of sets of the form
$$S_{a,b} = \{ < \in \LO(G) \mid a < b \}$$
where $a,b \in G$.  Here $<$ is to be considered as a total order in the usual way, and not as a function on triples.

\begin{prop}\label{prop:LOembeds}
The inclusion map $i: \LO(G) \into \CO(G)$ is an embedding, where $\LO(G)$ is topologized as by Sikora and $\CO(G)$ is topologized as a subspace of $\M$.
\end{prop}

\begin{proof}
By Lemma \ref{lem:COsubbasis}, $\CO(G)$ has a basis given by sets of the form
$$B_{T_1,\dots,T_n} := \{ c \in \CO(G) \mid c(T_j) = +1 \}$$
where $T_j := (x_j,y_j,z_j) \in G^3$.  Then
$$i^{-1}(B_{T_1,\dots,T_n})
= \{ < \in \LO(G) \mid c_<(T_j) = +1 \} 
= \bigcap_j \left( S_{x_j,y_j,z_j} \cup S_{y_j,z_j,x_j} \cup S_{z_j,x_j,y_j} \right)$$
where
$$S_{x,y,z} := S_{x,y} \cap S_{y,z}.$$
This shows $i$ is a continuous injection. Theorem 1.4 of \cite{Sikora04} says that $\LO(G)$ is a compact (and totally disconnected) space. This suffices to prove $i$ is an embedding, since a continuous injective function on a compact set is always an embedding (\ie a homeomorphism onto its image).
\end{proof}
%To prove $i$ is an embedding requires a little more work.  As follows from the results in the next section, in general $i$ will not be an open map.  But it is closed, which we now show.  Since a closed and injective continuous map is an embedding, the result will follow.

%Any closed set $C$ in $\LO(G)$ can be written
%$$C = \bigcap_{j} \bigcup_{i=1}^{I_j} S_{x_{j,i},y_{j,i}} = \bigcap_{j} \bigcup_{i=1}^{I_j} \{< \in \LO(G) \mid x_{j,i} < y_{j,i}\}$$
%where $I_j$ is a positive integer depending on $j$, and $x_{j_i},y_{j_i} \in G$.  If $< \in \LO(G)$ and $c_< = i(<)$, then $a<b$ if and only if $c_<(b^{-1}a,1,a^{-1}b) = 1$.  Thus
%$$i(C) = \LO(G) \cap \bigcap_{j} \bigcup_{i=1}^{I_j} \{c \in \CO(G) \mid c(y_{j,i}^{-1}x_{j,i},1, x_{j,i}^{-1}y_{j,i})=1\}.$$
%Since all of the sets involved in this intersection are closed in $\CO(G)$, we conclude $i$ is closed.

Next, we show $\LO(G)$ is closed inside of $\CO(G)$.  To do so, we shall find a simple criterion to tell which circular orderings are not genuine.  Our starting point is the following characterization of left-invariant orders on $G$:

\begin{lem}\label{lem:positivecone}
A group $G$ admits a left-invariant order if and only if there is a disjoint partition of $G = N \cup \{e\} \cup P$ such that $P \cdot P \subset P$ and $P^{-1} = N$ 
\end{lem}

\begin{proof} 
If $G$ has a left-invariant order $<$, then we can set $P = \{g\in G : g > e\}$. For the converse, we can define an order by $h<g$ if and only if $h^{-1}g \in P$ for all $g, h \in G$. 
\end{proof}

For a linear order $<$ on $G$, the set $P = \{g\in G : g > e\}$ is \emph{the positive cone} of $G$ with respect to the order $<$.  We just observed that a left-invariant positive cone characterizes a left-invariant order. This can be used to characterize the non-genuine circular orderings on $G$. 

\begin{prop}\label{thm:linearcharacterization}
Let $c$ be a circular ordering on a group $G$. Then $c$ is linear if and only if $c$ satisfies $c(h^{-1}g^{-1}, e, gh) = 1$ whenever both $c(h^{-1}, e, h)$ and $c(g^{-1}, e, g)$ are $1$. 
\end{prop} 

\begin{proof} 
Suppose a circular ordering $c$ is induced by a linear order $<$ on $G$ and let $P$ be the positive cone of $G$ with respect to $<$. By definition, $c(h^{-1}, e, h) = 1$ if either $h^{-1} < e< h$, $e < h < h^{-1}$ or $h < h^{-1} < e$. But since $<$ is a linear order, we know that if $e < h$, ie., $h \in P$, then we must have $h^{-1} < e$. Hence the only possibility is $h^{-1} < e< h$ and we can conclude that $c(g^{-1}, e, g) = 1$ if and only if $g \in P$ for all $g \in G$. Since $P$ is invariant under left-multiplication, if $c(h^{-1}, e, h) = c(g^{-1}, e, g) = 1$, then $c(h^{-1}g^{-1}, e, gh) = 1$. 

Conversely, suppose a circular ordering $c$ on $G$ has the property that $c(h^{-1}g^{-1}, e, gh) = 1$ whenever both $c(h^{-1}, e, h)$ and $c(g^{-1}, e, g)$ are $1$. Then one can define a positive cone $P$ by $\{g \in G : c(g^{-1}, e, g) = 1\}$. Hence one obtains a left-invariant order by setting $h<g$ if and only if $h^{-1}g \in P$.
\end{proof} 

It is straightforward to show that the condition of the proposition is a closed condition in $\CO(G)$.  We conclude

\begin{cor}\label{cor:LOclosedinCO}
$\LO(G)$ is closed in $\CO(G)$. \qed
\end{cor}

In particular, the proof of Corollary \ref{cor:CORE} can be modified to show

\begin{cor}
\label{cor:LORE}
There is an algorithm that, when input a finitely presented group $G = \langle S \mid R\rangle$, outputs (in finite time) $\mathsf{NO}$ together with an obstruction certificate if $G$ is not linearly orderable; if $G$ is linearly orderable, the algorithm runs forever.  In particular, non-linear-orderability is a recursively enumerable property. \qed
\end{cor}

Since $\LO(G)$ is closed in $\CO(G)$, one might wonder if it is also open.  In Section \ref{sec:abeliangroups} we will see that $\LO(\ZZ^n)$ is very far from being open, as $\CO_g(\ZZ^n)$ is dense in $\CO(\ZZ^n)$.  We pose the following general

\begin{ques}
For a given group $G$, what is the limit set of $\CO_g(G)$ inside $\CO(G)$? In particular, are there examples of groups where the limit set of $\CO_g(G)$ is not all of $\LO(G)$?
\end{ques}

For example, Navas and Rivas have shown that the set of bi-invariant linear orders on Thompson's group $F$ consists of 8 isolated points and 4 Cantor sets \cite{NavasRivas10}.  Are these 8 exotic linear orders still isolated in $\CO(F)$?

%%%%%%%%%%%%
%%%%%%%%%%%%
\subsection{$\Aut(G)$ action on $\CO(G)$}\label{sec:autaction}
Let $\Aut(G)$ be the set of automorphisms of $G$.  There is a natural left action of $\Aut(G)$ on $\CO(G)$.

\begin{lem}\label{lem:autactiondef}\label{lem:autactionhomeo}
Let $\rho \in \Aut(G)$ and $\phi \in \CO(G)$. Define a map $\rho \cdot \phi : G^3 \to \{-1, 0, 1\}$ by
$$ \rho \cdot \phi(g_0, g_1, g_2) = \phi (\rho^{-1}(g_0), \rho^{-1}(g_1), \rho^{-1}(g_2)), \forall g_0, g_1, g_2 \in G. $$
Then $\rho \cdot \phi$ is again in $\CO(G)$.  Moreover, $\Aut(G)$ acts on $\CO(G)$ by homeomorphisms.
\end{lem}

\begin{proof}
It is straightforward to check $\rho \cdot \phi \in \CO(G)$.  It is also straightforward to check that $\Aut(G)$ acts by homeomorphisms once we know $\phi \mapsto \rho \cdot \phi$ is a continuous map, so we show this.

Suppose $(\phi_i)$ is a sequence in $\CO(G)$ converging to $\phi \in \CO(G)$. Setting $A_i := \{ (g_0, g_1, g_2) \in G^3 : \phi_i(g_0, g_1, g_2) = \phi(g_0, g_1, g_3)\}$, $\phi_i \to \phi$ is equivalent to that the sequence $(A_i)$ has an increasing subsequence $(A_{i_j})$ so that $\cup_j A_{i_j} = G^3$. For any $\rho \in \Aut(G)$, it is clear that $(\rho^{-1}(A_{i_j}))$ is an increasing sequence so that $\cup_j \rho^{-1}(A_{i_j}) = \rho^{-1}(G^3) = G^3$. Furthermore, $\rho^{-1}(A_i) = \{ (g_0, g_1, g_2) \in G^3 : \rho\cdot \phi_i(g_0, g_1, g_2) = \rho\cdot \phi(g_0, g_1, g_3)\}$. Hence $\rho\cdot \phi_i$ converges to $\rho\cdot \phi$.
\end{proof} 

Let $\BCO(G)$ be the set of bi-invariant circular orders on $G$, \ie the circular orders which are invariant under both left and right mulplication.  It is easy to see that $\Inn(G)$ acts trivially on $\BCO(G)$, hence the $\Aut(G)$ action on $\CO(G)$ induces an action of $\Out(G)$ on $\BCO(G)$.

Further fruitful analysis of the action of $\Aut(G)$ on $\CO(G)$ would likely require restrictions on $G$.  One result in this direction is

\begin{prop}\label{thm:koberda}
Let $G$ be a residually torsion-free nilpotent group. Then the map $\Aut(G) \to \Homeo(\CO(G))$ is injective.
\end{prop} 

\begin{proof} 
Theorem 1.1 of \cite{Koberda11} states that the action $\Aut(G)$ on $\LO(G)$ is faithful, given such a $G$.   Since $\LO(G)$ is an invariant subset of $\CO(G)$, the result follows.
\end{proof} 

Observe that Koberda's theorem provides a criterion for left-orderability.  Indeed, for any residually torsion-free nilpotent group $G$ with non-trivial $\Aut(G)$, it follows that $G$ is left-orderable.  But, since our theorem is obtained as a corollary of Koberda's, it does not provide a similar criterion for genuine circular orderability.  In particular, it says nothing about whether $\Aut(G)$ acts on $\CO_g(G)$ faithfully, and fails for groups with nontrivial torsion.  In this direction, Corollary \ref{cor:abeliangenuinefaithful} of the next section shows that when $A$ is a finitely generated abelian group that is circularly orderable, $\Aut(A)$ acts faithfully on $\CO_g(A)$.

%%%%%%%%%%%%
%%%%%%%%%%%%
\section{Finitely generated abelian groups}\label{sec:abeliangroups}
In this section we will describe the space of circular orders $\CO(A)$ of a finitely generated abelian group $A$.  We assume the torsion subgroup of $A$ is cyclic, since otherwise $A$ is not circularly orderable.

\begin{lem}\label{lem:torsionandCO}
Let $G$ be a circularly orderable group with finite torsion subgroup $T \leq G$. Then $T$ is a finite cyclic group.
\end{lem}

\begin{proof}
Any circular order on $G$ restricts to a circular order on $T$.  Thus it suffices to show that the only circularly orderable finite groups are cyclic groups.  The circular order on $G$ restricts to yield an action of $T$ on $S^1$ with some trivially stabilized point $p_0 \in S^1$.  Let $\calo=\{p_0,\dots,p_{n-1}\}$ be the orbit of $p_0$ under $T$. The order structure on $T$ is encoded in its order-preserving action on this finite circularly ordered set $\calo$.  We can think of $\calo$ as the oriented cyclic digraph $C_n$, and the action of $T$ is by digraph automorphisms.  But the automorphism group of $C_n$ is isomorphic to $\mathbb{Z}/n\mathbb{Z}$.  Thus $T$ injects into $\mathbb{Z}/n\mathbb{Z}$.  Since $\calo$ is the orbit of $p_0$ under $T$, this map is surjective as well, so we conclude $T \simeq \mathbb{Z}/n\mathbb{Z}$.
\end{proof}

In contrast to this proposition, results in the next subsection imply that every finitely generated abelian group with cyclic torsion subgroup is circularly orderable.  By abuse of language, such a group will be called a \emph{circularly-orderable abelian group}.

We remark that the assumption of the lemma that the torsion subgroup $T$ be finite is necessary.  Indeed, consider the Pr\"ufer group
$$T = \lim_{n\rightarrow \infty} \ZZ/p^n\ZZ.$$
This is a torsion group with a natural rotation action on $S^1$.  In particular, $T$ is circularly orderable.  As for finitely generated, infinite torsion groups--so called \emph{Burnside groups}--the conclusion of the lemma is also known to hold \cite{Navas11}. In general, the problem of deciding when the homeomorphism group of a manifold has the Burnside property is open, although recent progress has been made in the case of compact surfaces \cite{GuelmanLiousse14}.

%%%%%
\subsection{Classifying elements of $\CO(A)$}\label{ss:abelianclassification}
For the rest of the section, $A$ will denote a circularly-orderable abelian group. The first step in our description of $\CO(A)$ is to use two known results to construct a recursive classification of the elements.  For background and development of both these results, the reader is referred to Ghys's highly readable paper \cite{Ghys01}.  The first result goes back to Poincar\'e and involves a careful analysis of rotation numbers.

\begin{prop}[Proposition 5.6 of \cite{Ghys01}]\label{prop:Poincare}
Let $G$ be any subgroup of $\Homeop(S^1)$. Then there are three mutually exclusive possibilities.\begin{enumerate}
\item There is a finite orbit.
\item All orbits are dense.
\item There is a compact $G$-invariant subset $C \subset S^1$ which is infinite and different from $S^1$ and such that the orbits of points in $C$ are dense in $C$. This set $C$ is unique, contained in the closure of any orbit and is homeomorphic to 
a Cantor set. $C$ is called the \emph{exceptional minimal set} of $G$.
\end{enumerate}
\end{prop}

In both cases 1 or 3, we associate to $G$ the stabilizer subgroup $K$ of the finite orbit or the Cantor set $C$, respectively.  We call $K$ the \emph{blowdown kernel}.

In case 3, we can blow down the gaps of $S^1 \setminus C$ to get an action of type 2.  That is, we construct a quotient of $S^1$ so that the closure of each maximal, connected, open interval of $S^1 \setminus C$ is replaced by a point. The resulting space is homeomorphic to $S^1$, and the fact that the action is of type 2 follows from

\begin{prop}[Proposition 5.8 of \cite{Ghys01}]
Let $G$ be a group and $r: G \to \Homeop(S^1)$ a homomorphism such that $r(G)$ has an exceptional minimal set $K$. Then there is a homomorphism $\overline{r}: G \to \Homeop(S^1)$ such that $r$ is semi-conjugate to $\overline{r}$ and $\overline{r}(G)$ has dense orbits on the circle. \qed
\end{prop}

We call a subgroup $G$ of $\Homeop(S^1)$ \emph{minimal} if all of its orbits are dense.  Similarly, we say a circular order is minimal if its dynamical realization is minimal.  Let $\COmin(G)$ denote the set of minimal circular orders of $G$, let $\COfin(G)$ be the set of circular orders of $G$ with a finite orbit, and let $\COb(G)$ be the set of circular orders of type 3.

The second result we exploit is a sort of Tits alternative due to Margulis \cite{Margulis00}.

\begin{thm}[\cite{Margulis00}, see also Corollary 5.15 of \cite{Ghys01}]\label{thm:Tits}
Let $G$ be a subgroup of $\Homeop(S^1)$ such that all orbits are dense in the circle. Exactly one of the following properties holds:
\begin{enumerate}
\item $G$ contains a non abelian free subgroup.
\item $G$ is abelian and is conjugate to a group of rotations.
\end{enumerate}
\end{thm}

Finally, a piece of notation.  Let $\TT^n = S^1 \times \dots \times S^1$ denote the $n$-torus, where we identify the circle $S^1 = [0,1]/\{0=1\}$ with the unit interval with endpoints glued together.  An \emph{irrational point} on the torus $\TT^n$ is a point such that each coordinate is irrational.  A \emph{totally irrational point} is an irrational point such that all of the coordinates are pairwise noncommensurable, i.e.\ their ratios are not rational.  We denote the set of totally irrational points on $\TT^n$ by $T_n$.

We are now well positioned to classify the set of circular orders on $\ZZ^n$.

\begin{thm}\label{thm:freeabelianCO}
The set $\CO(\ZZ^n)$ of left invariant circular orders on $\ZZ^n$ is a disjoint union of $\COfin(\ZZ^n)$, $\COmin(\ZZ^n)$ and $\COb(\ZZ^n)$, with
$$\begin{aligned}
\COfin(\ZZ^n) &= \bigsqcup_{\substack{K \leq \ZZ^n \\ \rank K = n}}\left[\LO(K) \times \CO(\ZZ^n /K)\right], \\
\COb(\ZZ^n) &=  \bigsqcup_{\substack{K\leq \ZZ^n\\ \rank K < n}}\left[ \LO(K) \times \CO(\ZZ^n/K) \right] \text{, \ \ and}\\
\COmin(\ZZ^n) &= T_n	.\end{aligned}$$
In particular,
$$\CO_g(\ZZ^n) = \bigsqcup_{\substack{K < \ZZ^n \\ K \neq \ZZ^n \\ \rank K = n}}\left[\LO(K) \times \CO(\ZZ^n /K)\right] \cup \COmin(\ZZ^n) \cup \COb(\ZZ^n).$$
\end{thm}

The reader will notice that we have referred to $\CO(\ZZ^n/K)$ without describing it.  Of course, for $K$ full rank, $\ZZ^n/K$ is finite, and hence $\CO(\ZZ^n/K)$ is finite (possibly empty), consisting of ``rotation orders."  For $\rank(K) < n$, $\CO(\ZZ^n/K)$ will be empty if the quotient has a noncyclic torsion subgroup; otherwise, $\ZZ^n/K = \ZZ^{n-\rank K} \times \ZZ/m$ for some integer $m$, in which case we give a description of $\CO(\ZZ^n/K)$ in the next theorem. So by combining Theorems \ref{thm:freeabelianCO} and \ref{thm:abelianCO}, we will in fact have a kind of recursive classification.

\begin{proof}
The first statement, that $\CO(\ZZ^n)$ is a disjoint union of $\COfin,$ $\COmin$ and $\COb$ follows from the definitions and the fact that the three types of subgroups of Proposition \ref{prop:Poincare} are mutually exclusive.  For the three middle statements, we will analyze the possiblities for a dynamical realization $r_c$ of a circular order $c$ on $\ZZ^n$.  By Proposition \ref{prop:Poincare}, there are three cases to consider.

\begin{enumerate}
\item First suppose $r_c$ has a finite orbit $\mathcal{O}=\{p_1,\dots,p_m\}$ with $p_m < p_1 < \dots < p_m$.  Then there is an induced order preserving action $\sigma$ of $\ZZ^n$ on $\mathcal{O}$.  This is equivalent to the homomorphism $\sigma: \ZZ^n \to \ZZ/m$ where the kernel is
$$K:=\ker\sigma = \bigcap_{i=1}^m \Stab_{\ZZ^n}(p_i),$$
\ie the blowdown kernel.  Since $K$ fixes all of the $p_i$ and consists of order preserving maps, $K$ maps the intervals $I_1,\dots,I_n$ of $S^1 \setminus\mathcal{O}$ to themselves.  Our action is a dynamical realization action, so there is a marked trivially stabilized point $q$ in one of the $I_i$ such that the order structure of $q$ is $c$. Now consider the restricted action of $K$ on $I_i \simeq \RR$.  The orbit structure of $q$ under $K$ in $I_i$ induces a linear order on $K$, since $I_i$ maps to itself by $K$.  Since $\ZZ^n/K \simeq \ZZ/m$, $K$ must be full rank, i.e.\ $K\simeq \ZZ^n$.  This construction yields a map
$$\COfin(\ZZ^n) \to \bigsqcup_{\substack{K \leq \ZZ^n \\ \rank K = n}}\left[\LO(K) \times \CO(\ZZ^n /K)\right].$$

\

Inversely, for any full rank subgroup $K < \ZZ^n$ of index $m$, we can intertwine a linear order on $K$ with a circular order on $\ZZ^n/K$.  More specifically, let $\ZZ^n/K \simeq \ZZ/m$ act by the rotation action indicated by the circular order, pick a point $p$, replace each point in the orbit of $p$ under $\ZZ/m$ with identical blown-up intervals, and let $K$ act simultaneously and identically on these $m$ intervals $I_1,\dots,I_m$ by the dynamical realization of the linear order on $K$, with marked trivially stabilized point $q$.  Since $\Homeop(\RR)$ is divisible, we can extend the $K$ action on $S^1$ to an action of $\ZZ^n$ by rotating $S^1$ and acting inside the blow-up intervals accordingly.  Clearly the order structure of $q$ is inverse to the construction of the previous paragraph, so we have established a bijection between orders with finite orbits and
$$\bigsqcup_{\substack{K \leq \ZZ^n \\ \rank K = n}}\left[\LO(K) \times \CO(\ZZ^n /K)\right].$$

\item Now suppose $r_c$ is minimal. Theorem \ref{thm:Tits} implies the image of $\ZZ^n$ in $\Homeop(S^1)$ is conjugate to a group of rotations.  Since $r_c$ is an injection, we need the generators of $\ZZ^n$ to map to irrational rotations that are all pairwise incommensurable.  Of course, if the collection of images of the generators are distinct, the circular orders are distinct. Conversely, any assignment of incommensurable irrational rotations to the standard generators of $\ZZ^n$ yields an injection into $\Homeop(S^1)$ with trivial stabilizers.  This explains the component of $\CO(\ZZ^n)$ parametrized by $T_n$.

\item Finally, suppose $r_c$ is of type 3, so there is an invariant Cantor set $C$.  Let $q$ be the marked trivially stabilized point of the dynamical realization.  Blow $r_c$ down to $C$ via $b$ to get a new action $r = b \circ r_c$, which may no longer be faithful.  The blowdown kernel $K$ stabilizes $b(q)$, hence $K$ acts on the interval $b^{-1}(b(q))$ in a way such that $q$ is trivially stabilized.  Thus $c$ restricts to a linear order on $K$.  Of course, $\ZZ^n/K$ acts faithfully on $S^1$ via $r$ and trivially stabilizes $b(q)$.  Thus we get a circular order on $\ZZ^n/K$.  Moreover,  Proposition \ref{prop:Poincare} implies all $r$ orbits are dense in $b(S^1)=S^1$, so the action of $\ZZ^n/K$ on $S^1$ is minimal.  In particular, we conclude $\rank K < n$.  Thus every type 3 circular order of $\ZZ^n$ yields a rank subgroup $K$ not of full rank, together with a linear order on $K$ and a circular order on the quotient.  Moreover, given such data, by the same procedure in the case of a finite orbit, we can construct a circular order on $\ZZ^n$ inverse to this.
\end{enumerate}

The last statement about the genuine orders is clear.
\end{proof}

The arguments given in the proof of Theorem \ref{thm:freeabelianCO} to classify $\COfin(\ZZ^n)$ and $\COb(\ZZ^n)$ are fairly general.  As such, there are clear generalizations of these descriptions for arbitrary groups.  On the other hand, the classification of $\COmin(\ZZ^n)$ is highly dependent on Margulis's Theorem \ref{thm:Tits} and the fact that our group is abelian.  We see that by applying this theorem to $\ZZ^n \times \ZZ/m$ we can similarly classify $\CO(\ZZ^n \times \ZZ/m)$.

\begin{thm}\label{thm:abelianCO}
The set $\CO(\ZZ^n \times \ZZ/m)$ of left invariant circular orders on $\ZZ^n \times \ZZ/m$ is a disjoint union of $\COfin(\ZZ^n \times \ZZ/m)$, $\COmin(\ZZ^n \times \ZZ/m)$ and $\COb(\ZZ^n \times \ZZ/m)$, with
$$\begin{aligned}
\COfin(\ZZ^n \times \ZZ/m) &= \bigsqcup_{\substack{K \leq \ZZ^n \times \ZZ/m\\ \rank K = n}}\left[\LO(K) \times \CO[(\ZZ^n \times \ZZ/m) /K]\right], \\
\COb(\ZZ^n \times \ZZ/m) &=  \bigsqcup_{\substack{K\leq \ZZ^n \times \ZZ/m\\ \rank K < n}}\left[ \LO(K) \times \CO[(\ZZ^n \times \ZZ/m)/K] \right] \text{, \ \ and}\\
\COmin(\ZZ^n \times \ZZ/m) &= T_n \times \CO(\ZZ/m).	\end{aligned}$$
Since $\ZZ^n \times \ZZ/m$ has torsion, $\LO(\ZZ^n \times \ZZ/m) = \emptyset$, \ie $\CO_g(\ZZ^n \times \ZZ/m) = \CO(\ZZ^n \times \ZZ/m)$. \qed
\end{thm}

We should say a word about why we can sensibly consider Theorems \ref{thm:freeabelianCO} and \ref{thm:abelianCO} as classification theorems.  Indeed, by the previous remark and Margulis's theorem, we could give a superficially similar classification for $\CO(G)$ for any $G$ with no nonabelian free subgroups.  However, for arbitrary $G$, determining and describing possible blowdown kernels is a hard problem.  While we have admittedly not explicitly described the solution to this problem for a finitely generated abelian group $A$, it is clear that the blowdown kernels in this case are precisely subgroups $K \leq A$ such that $A/K$ is a free abelian group with cyclic torsion subgroup.  Furthermore, for arbitrary $G$ and blowdown kernel $K \leq G$, we do not necessarily have a good understanding of dense subsets of $\LO(K)$ or $\CO(G/K)$.  But for a finitely generated abelian group $A$, we do, which is the topic of the next subsection.

%%%%%
\subsection{Topology of $\CO(A)$}\label{ss:abeliantopology}
We now want to understand how $\COfin(A)$, $\COmin(A)$ and $\COb(A)$ relate to the topology of $\CO(A)$.  For convenience, we introduce new notation: $\COrot(A)$ consists of orders which have a dynamical realization that is a rotation action.  Thus, for $A = \ZZ^n \times \ZZ/m$ with $n>0$, $\COrot(A) = \COmin(A)$, and for $A$ with $\rank(A) = 0$, $\COrot(A) = \CO(A)$.  For $c \in \COrot(A)$, we will write $c = c_\theta$ for $\theta \in T_n$ if $A = \ZZ^N$ has no torsion, $c= c_k$  if $A = \ZZ/m$ and $c = c_{\theta,k}$ for $A = \ZZ^n \times \ZZ/m$.  Here $0\leq k <m$ is an integer coprime to $m$, and indicates that a fixed generator of $\ZZ/m$ (say, $\overline{1}$) acts by a rotation of angle $\frac{k}{m}$ (recall that our circle is $S^1 = [0,1]/\{0=1\}$).  By abuse of notation, even if $A$ has no rank or torsion, we may write $c = c_{\theta,k}$, in which case we simply ignore $\theta$ or $k$, respectively.

Our goal now is to show that $\COrot(A)$ is dense in $\CO(A)$.  We will conclude as a corollary that when $\rank(A) >0$ and the torsion subgroup of $A$ is cyclic, $\CO(A)$ is a Cantor set and $\LO(A)$ is not open.

First we isolate convenient dense subsets of $\COfin(A)$ and $\COb(A)$.  By the proof of Theorem \ref{thm:abelianCO}, we know for every $c \in \COfin(A) \cup \COb(A)$ there is a blowdown kernel $K \leq A$, a linear order on $K$ and a circular order on $A/K$ such that $c$ is constructed by intertwining the two.  Given such a $K$, we let $\LOtran(K)$ be the subspace of $\LO(K)$ consisting
of ``translation orders.'' More precisely, suppose $K \simeq \ZZ^k$ is rank $k$ and consider the set
$$S_{k-1} = \{(x_1,\dots,x_k)\in (\RR^\times)^k \mid \frac{x_i}{x_j} \notin \mathbb{Q} \ \forall i\neq j \} / \RR_+$$
where $\RR_+$ acts diagonally by scaling.  We say $c \in \LO(K)$ is a \emph{translation order} if it has a dynamical realization such that the $i^\text{th}$ standard generator (under some fixed identification $K = \ZZ^k$) acts by translation by a distance $x_i$.  $S_{k-1}$ could be loosely considered as the totally irrational points on the sphere $S^{k-1}$, hence is an $\LO$ analog to the totally irrational rotation circular orders. Notice we get the sphere $S^{k-1}$ and not $\mathbb{RP}^{k-1}$, since the action is by $\RR_+$ and not $\RR\setminus\{0\}$.

In \cite{Koberda11}, Koberda constructs a subspace of $\LO(\ZZ^n)$ we shall call $\LOflag(\ZZ^n)$, consisting of positive cones constructed from rational flags in $\ZZ^n$.  He then shows $\LOflag(\ZZ^n)$ is dense in $\LO(\ZZ^n)$.  The interested reader can easily check that the limit set of $\LOtran(\ZZ^n)$ in $\LO(\ZZ^n)$ contains $\LOflag(\ZZ^n)$.  The important fact for our purposes is that we can deduce

\begin{lem}\label{lem:LOtran}
$\LOtran(\ZZ^n)$ is dense in $\LO(\ZZ^n)$. \qed
\end{lem}

Now define
$$\COfintran(A) := \bigsqcup_{\substack{K \leq A \\ \rank K = n}} \LOtran(K) \times \CO(A / K) \subset \COfin(A)$$
and
$$\CObtran(A) := \bigsqcup_{\substack{K \subset A\\ \rank K < n}} \LOtran(K) \times \CO(A /K) \subset \COb(A)$$
to be subspaces consisting of orders whose restriction to their blowdown kernels are translation orders. By the previous lemma, we have

\begin{lem}\label{lem:COfintran}
$\LOtran(K) \times \CO(A / K)$ is dense in $\LO(K) \times \CO(A / K)$.  In particular, $\COfintran(A)$ is dense in $\COfin(A)$ and $\CObtran(A)$ is dense in $\COb(A)$. \qed
\end{lem}

We can now establish our key result for understanding the topology of $\CO(A)$.

\begin{thm}\label{thm:rotationordersdense}
For any finitely generated abelian group $A$, the space $\COrot(A)$ of rotation orders of $A$ is dense in the space $\CO(A)$ of all circular orders.
\end{thm}

\begin{proof}
The proof proceeds via direct analysis and Noetherian induction.  We consider the poset of finitely generated abelian groups with cyclic torsion subgroups where $A > B$ if there is a surjection $A \onto B$.  Of course, the base cases are cyclic groups of prime order, but the statement of the theorem is trivial for all cyclic groups because $\COrot(A) = \CO(A)$. We will show in separate cases that $\COmin(A)$ is dense in both $\COfin(A)$ and $\COb(A)$.

First we will show $\COrot(A)$ is dense in $\COfin(A)$.  For this case, we will not need induction.  Let $c \in \COfin(A)$ and let $S \subset A \times A \times A$ be a finite subset of triples.  Write
$$A = \ZZ \times \dots \times \ZZ \times \ZZ/m_0.$$
We will find $\theta' \in T_n$ and $k$ with $(k,m_0) = 1, 0 \leq k < m_0$ such that $c_{\theta',k} |S = c|S$.  Let $\pi S$ be the set of elements of $A$ involved in the triples of $S$, \ie
$$\pi S := \{a \in A \mid \exists b,c \in A \text{ with } (a,b,c) \in S, (b,a,c) \in S \text{ or } (b,c,a) \in S \}.$$
Let
$$\epsilon = \min_{\substack{a,b \in \pi S \\ a \neq b}} | (r_{c}(a))(0) - (r_{c}(b))(0) |.$$
It suffices to find $\theta \in T_n$ and $k$ with $(k,m_0) = 1, 0 \leq k < m_0$ such that
$$ | (r_{c_{\theta,k}}(a))(0) - (r_{c}(a))(0) | < \frac{\epsilon}{2}$$
for all $a \in \pi S$.  Furthermore, by Lemma \ref{lem:COfintran} it suffices to assume $c \in \COfintran(A)$.  We will give an explicit description of a dynamical realization of such a $c$ and do some simple analysis with it.

Let the blowdown kernel of $c$ be
$$K = m_1\ZZ \times \dots \times m_n\ZZ \times m_{n+1}\ZZ/m_0\leq A,$$
so that $c \in \LOtran(K) \times \CO(A/K)$.  Note that $\LO(K)$ is empty if $K$ has torsion, so we assume
$$K = m_1\ZZ \times \dots \times m_n\ZZ < A.$$

Let the translation order of $c|K$ have a dynamical realization encoded by the translation data $[(m_1x_1,\dots,m_nx_n)] \in S_{n-1}$, meaning $m_ie_i \in K$ acts by the translation of length $m_ix_i$.  Let $x = (x_1,\dots,x_n)$.  Since
$$A/K = \ZZ/m_1 \times \dots \times \ZZ/m_n \times \ZZ/m_0$$
is assumed to be cyclic, we fix an isomorphism
$$A/K \to \ZZ/M$$
where $M := m_1\dots m_nm_0$ by letting $\overline{e_i} \mapsto \frac{M}{m_i} \pmod M$.  Let the blown-down circular order on $A/K$ have a dynamical realization in which $\overline{ e_1 + \dots + e_n + e_0} = 1 \pmod M$ acts by the rotation
$$\frac{k}{M}$$
where $(k,M) = 1$.  Thus $e_i$ acts by the rotation $r_i = \frac{k}{m_i}$.  Encode these rotation angles in the vector $r = (r_1,\dots,r_n,r_0)$.  Let $a = a_1e_1 + \cdots a_ne_n + a_0e_0 = (a_1,\dots,a_n,a_0) \in \ZZ^n \times \ZZ/m_0$ and $a' = (a_1,\dots,a_n)$.  Then we can describe a dynamical realization of $c$ by
$$(r_c(a))(0) = T_{a' \cdot x} \circ R_{a \cdot r}(0)  \mod 1$$
where $a\cdot r$ and $a' \cdot x$ are dot products of vectors, $R_{a\cdot r}$ is a rotation by angle $a\cdot r$ and the support of $T_{a' \cdot x}$ is the union of the blow-up intervals, on each of which $T_{a' \cdot x}$ acts by the translation of distance $a' \cdot x$.  More explicitly, there are $\frac{1}{M}$ blow-up intervals $I_i$, each of length $\ell = \frac{1}{2M}$, and on each of these we might have
$$T_{a' \cdot x}(p) = h_i^{-1}\left( \frac{\ell}{\pi} \arctan \left[ \tan\left(\frac{\pi}{\ell}h_i(p)\right) + a' \cdot x\right]\right)$$
where $h_i: I_i \to (-\frac{\ell}{2},\frac{\ell}{2})$ is an orientation preserving isometry taking the midpoint of $I_i$ to $0$.  Such a requirement for the $h_i$'s is important, because it ensures $r_c$ actually yields a group action.

We compute
$$(r_c(a))(0) = a_1r_1 + \dots a_nr_n +a_0r_0 + \left( \frac{\ell}{\pi} \arctan \left( a' \cdot x\right)\right)
= a\cdot r + \left( \frac{\ell}{\pi} \arctan \left( a' \cdot x\right)\right).$$
We need to find $\theta' \in T_n$ close to $r'=(r_1,\dots,r_n)$ and in the linear regime of $\arctan$ as a function of all the $a'$ found in $\pi S$.  Write $\theta = r + (\delta,0)$, $\theta' = r' + \delta'$ where $\delta' = (\delta_1,\dots,\delta_n)$.  Then
$$| (r_{c_{\theta',k}}(a))(0) - (r_c(a))(0) | = |a\cdot \theta - a \cdot r - \left( \frac{\ell}{\pi} \arctan \left( a' \cdot x\right)\right)|
=  |a'\cdot \delta - \left( \frac{\ell}{\pi} \arctan \left( a' \cdot x\right)\right)|$$
The essential observation is that we can scale $x$ sufficiently small enough so that $\arctan \left( a' \cdot x\right)$ is approximately linear for all $a'$ involved in $\pi S$.  This is devious because we change the dynamical realization of $c$ depending on $\pi S$.

Now clearly we can pick $\delta_1,\dots,\delta_n$ so that $r + \delta \in S_{n-1}$ and
$$|a'\cdot \delta - \left( \frac{\ell}{\pi} \arctan \left( a' \cdot x\right)\right)| < \frac{\epsilon}{2}.$$
This shows $\COrot(A)$ is dense in $\COfin(A)$.

To show $\COrot(A)$ is dense in $\COb(A)$, we apply the induction hypothesis.  Specifically, induction tells us
$$\CObtranrot(A) = \bigsqcup_{\substack{K \subset A\\ \rank K < n}} \LOtran(K) \times \COrot(A /K) \subset \CObtran(A)$$
is a dense subset of $\CObtran(A)$, hence $\CObtranrot(A)$ is dense in $\COb(A)$.  But now a hands-on analytic argument can be applied exactly as above.  We spare the reader any more details.
\end{proof}

Finally, we can address the topology of $\CO(A)$.

\begin{thm}\label{thm:cantorset}
Suppose $\rank(A) >0$ and $A$ has a cyclic torsion subgroup.  Then $\CO(A)$ is a Cantor set.
\end{thm}

\begin{proof}
Let $\rank(A) = n$ and the torsion subgroup of $A$ have order $m$, so
$$A = \ZZ \times \dots \times \ZZ \times \ZZ/m.$$

We already know $\CO(A)$ is a closed subset of a Cantor set, so it suffices to show it is perfect.  In fact, since $\COrot(A)$ is dense, it suffices to show $\COrot(A)$ is perfect.  To check for perfectness, we ought to show that every basis open set is either empty or infinite.  So let $$B = B_{t_1,\dots,t_l} = \{ c \in \CO(A) \mid c(t_1) = \dots = c(t_l) = 1\}$$ be a nonempty basis set, where $t_1,\dots,t_l \in A^3$ is a finite set of triples of $A$.  Then by density there is some $c_{\theta,k} \in B$ with $c_{\theta,k} \in \COrot(A)$, where $\theta = (\theta_1,\dots,\theta_n) \in T_n$ and $(k,m) = 1, 0 \leq k < m$.  Recall that $k$ indicates that the standard generator of $\ZZ/m$ acts by rotation $\frac{k}{m}$.  Write $S = \{ t_1,\dots,t_l\}$ and let
$$\pi S = \{a \in \ZZ^n \mid \exists b,c \in \ZZ^n \text{ with } (a,b,c) \in S, (b,a,c) \in S \text{ or } (b,c,a) \in S \}$$
be the set of elements of $A$ involved in $S$.  Suppose our marked trivially stabilized point for the dynamical realization $r_{c_{\theta,k}}$ of $c_{\theta,k}$ is $0 \in S^1$.  Let
$$\epsilon = \min_{\substack{a, b \in \pi S \\ a \neq b}} | (r_{c_{\theta,k}}(a))(0) - (r_{c_{\theta,k}}(b))(0) |$$
be the smallest distance between points in the orbit of 0 under $r_{c_{\theta,k}}$ restricted to elements of $\ZZ^n$ involved in $S$.  Note this minimum makes sense because $S$ is finite.  In particular, $0 < \epsilon < 1$.

By the triangle inequality, it suffices to find infinitely many $\psi = (\psi_1,\dots,\psi_n) \in T_n$ such that
$$| (r_{c_{\theta,k}}(a))(0) - (r_{c_{\psi,k}}(a))(0) | < \frac{\epsilon}{2}$$
for all $a \in \pi T$.

Notice that we can write
$$(r_{c_{\theta,k}}(a))(0) = a \cdot (\theta,\frac{k}{m})$$
where $\cdot$ is the dot product and $a \in A = \ZZ \times \dots \times \ZZ \times \ZZ/m$.  Now we can use the Cauchy-Schwarz inequality to conclude
$$ | (r_{c_{\theta,k}}(a))(0) - (r_{c_{\psi,k}}(a))(0) | = | a \cdot (\theta - \psi,0) |  \leq \| a\|_2 \| \theta - \psi\|_2 < C \|\theta - \psi\|_2$$
where $C$ is some finite constant depending only on $\pi S$.  Of course, we can find infinitely many $\psi \in T_n$ such that
$$\|\theta - \psi\|_2 < \frac{\epsilon}{2C}$$
so the claim follows.
\end{proof}

The density of rotation orders can also be used to show

\begin{cor}\label{cor:abeliangenuinefaithful}
Let $A$ be a finitely generated abelian group. Then $\Aut(A)$ acts faithfully on $CO_g(A)$. \qed
\end{cor} 

Indeed, this corollary follows immediately from the fact that $CO_{rot}(A) \subset \CO_g(A)$ is dense and the following

\begin{thm}\label{prop:abelianrotfaithful}
Let $A$ be a finitely generated abelian group that is circularly orderable. Then $\Aut(A)$ acts faithfully on $CO_{rot}(A)$.  In fact, $\Aut(A)$ acts freely.
\end{thm}

\begin{proof}
Let $A = \ZZ^n \times \ZZ/m\ZZ$, with $m = 0$ corresponding to the case of $A$ free abelian.  Then
$$\Aut(A) \simeq \Aut(\ZZ^n) \oplus \Hom(\ZZ^n,\ZZ/m\ZZ) \oplus \Aut(\ZZ/m\ZZ)$$
where
$$\Hom(\ZZ^n, \ZZ/m\ZZ) \simeq (\ZZ/m\ZZ)^n.$$
To see this, let $\alpha$ be an automorphism of $A$.  Let $i_1: \ZZ^n \to A$ and $i_2: \ZZ/m\ZZ \to A$ be the natural inclusions, and let $\pi_1: A \to \ZZ^n$ and $\pi_2: A \to \ZZ/m\ZZ$ be the natural projections.  Then, by the universal properties of direct sum and direct product of abelian groups (which are the same), $\alpha$ is determined by the four maps
$$\pi_1 \circ \alpha \circ i_1: \ZZ^n \to \ZZ^n,$$
$$\pi_1 \circ \alpha \circ i_2: \ZZ/m\ZZ \to \ZZ^n,$$
$$\pi_2 \circ \alpha \circ i_1: \ZZ^n \to \ZZ/m\ZZ,$$
$$\pi_2 \circ \alpha \circ i_2: \ZZ/m\ZZ \to \ZZ/m\ZZ.$$
Of course, $\pi_1 \circ \alpha \circ i_2$ has to be trivial.  A simple check shows $\alpha$ is bijective if and only if $\pi_1 \circ \alpha \circ i_1$ and $\pi_2 \circ \alpha \circ i_2$ are.  In particular, $\pi_2 \circ \alpha \circ i_1$ can be any homomorphism.

Now let $c_{\vec{\theta},k}$ be an element of $\COrot(A)$, where $\vec{\theta} = (\theta_1,\dots,\theta_n) \in T_n$ and $k$ is coprime to $m$.  Recall the subscripts indicate that $c = c_{\vec{\theta},k}$ has a dynamical realization $r_c$ in which
$$r_c(a_1,\dots,a_n,a) = \vec{\theta}\cdot(a_1,\dots,a_n) + \frac{ka}{m} \in S^1 \leq \Homeop(S^1)$$
for all $(a_1,\dots,a_n,a) \in A$.  Then $\alpha$ has the following effect on the dynamical realization of $c$:
$$r_{\alpha \cdot c} = \vec{\theta}\cdot [(\pi_1 \circ \alpha^{-1})(a_1,\dots,a_n,a)] + \frac{k}{m}[(\pi_2 \circ \alpha^{-1})(a_1,\dots,a_n,a)].$$
To see this, note that because $c$ is a rotation order, the dynamical realization is conjugate to the (set-theoretic) order-preserving embedding $i_c: A \into S^1$ constructed in Proposition \ref{prop:faithfulactioniffCO}.  Thinking of $i_c$ as a way to label points in $S^1$ by elements of $A$, it is clear the affect of $\alpha$ on this labelling is to change the label $x$ to the label $\alpha^{-1}(x)$.

If $\alpha \cdot c = c$, then $r_{\alpha \cdot c} = r_c$.  Since $r_c$ is faithful, for all $(a_1,\dots,a_n,a)$ in $A$ conclude
$$\pi_1 \circ \alpha^{-1}(a_1,\dots,a_n,a) = (a_1,\dots,a_n)$$
and
$$\pi_2 \circ \alpha^{-1}(a_1,\dots,a_n,a) = a.$$
The components
$$\pi_1 \circ \alpha^{-1} \circ i_1$$
and
$$\pi_2 \circ \alpha^{-1} \circ i_2$$
of $\alpha$ are therefore identity maps, and the component
$$\pi_2 \circ \alpha^{-1} \circ i_1$$
must be the zero map.  Thus $\alpha^{-1}$, hence $\alpha$, is the identity map.
\end{proof}

\subsection{Archimedean Orders}\label{ss:archimedean}
In the theory of left-invariant linear orders, there are special classes of orders which have been isolated. For instance, a linear order $<$ on $G$ is called \emph{Archimedean} if for any $g, h \in G \setminus \{\Id\}$, there exists $n \in \ZZ$ such that $g^n > h$. We propose a similar notion for circular orders. 

For a countable group $G$ that is not infinite cyclic, we say a circular order $c$ is \emph{Archimedean} if for any two elements $g, h$ of $G$ which are not powers of the same element of $G$, there exists a positive integer $n$ such that
$$c(e, g, h) \neq c(e, g^n, h).$$
We exclude the case that $G = \ZZ$, since otherwise every order on $\ZZ$ would be Archimedean, and the results of this subsection would have to be modified.

The Archimedean property for circular orders is a generalization of the Archimedean property of linear orders, in the sense that if $c$ were a linear order and for all triples with $c(e,g,h) = +1$ (\ie $e < g < h$) there exists an $n>0$ such that the condition $c(e,g^n,h) = -1$ holds, then $c$ would be called an Archimedean linear order. On the other hand, the fact that our definition of Archimedean circular order does not require $c(e,g,h) = +1$ imposes more serious restrictions.

\begin{lem} 
\label{lem:Archimedeangenuine}
An Archimedean circular order is always genuine. 
\end{lem} 

\begin{proof}
Suppose $c$ is an Archimedean circular order of a group $G$ which is induced by a linear order $<$ of $G$. Take any $g, h \in G$ satisfying $\Id < h < g$ or $g < \Id < h$. But the Archimedean property for circular orders implies that $\Id < g^n < h$ for some power $n$, which is impossible. 
\end{proof} 

We close this section by answering a natural question: which circular orders are Archimedean? It turns out that as in the case of linear orders (see Section 3 of \cite{Navas10}), Archimedean circular orders arise from free actions.

\begin{prop}\label{prop:archimedeanfree} 
Let $G$ be a group that is not infinite cyclic. Then a circular order $c$ on $G$ is Archimedean if and only if the dynamical realization $r_c$ is a free action. 
\end{prop} 

\begin{proof}
We begin by observing that if $G$ is finite cyclic, then every order on $G$ is vacuously Archimedean, and every dynamical realization is a rotation action, hence free. Thus the proposition holds for $G$ finite cyclic. For the remainder of the proof, we assume $G$ is noncyclic.

Suppose $c$ is not Archimedean. Then there exists $g, h \in G$ which are not powers of the same element such that $c(p, r_c(g)(p), r_c(h)(p)) = c(p, r_c(g^n)(p), r_c(h)(p))$ for all $n > 0$. This implies that the orbit of $p$ under forward iterates of $g$ is completely contained in one of the connected components of $S^1 \setminus \{p, r_c(h)(p)\}$. Since $g$ is a homeomorphism, $r_c(g^n)(p)$ accumulates to a point, and such a point must be fixed by $g$. 

Conversely, suppose $c$ is Archimedean, and let $p \in S^1$. If $p$ is in the image of $G$ under the order-preserving embedding used to construct $r_c$, then $p$ must be trivially stabilized. So suppose $p$ is not in this image, and let $g \in G$. We will show $g$ does not fix $p$.

If $g$ is a torsion element, then $g$ is conjugate to a rotation, hence has no fixed points.  So we suppose $g$ is not torsion.

Since $G$ is not infinite cyclic, there exists $h\in G$ such that $h$ and $g$ are not in a cyclic subgroup. Let $n>1$ be the smallest positive integer such that
$$c(e,g,h) \neq c(e,g^n,h).$$

For convenience, we will assume $c(e,g,h) = +1$. Consider the four open intervals $(e,g), (g,h), (h,g^n)$ and $(g^n,e)$ comprising the connected components of $S^1 \setminus \{e,g,h,g^n\}$. Here, the notation $(x,y)$ means the component of $S^1 \setminus \{x,y\}$ such that $(x,p,y)$ is positively oriented for all $p$ in that component. We will consider four cases corresponding to which interval $I_i$ contains $p$. For now, we will assume $n > 2$.  The reader is highly encouraged to draw pictures for what follows.

\begin{enumerate}
\item If $p \in (e,g)$, we consider two subcases. If for all $1< i < n$, $g^i \in (p,g)$, then $p \in (e,g^i)$ for all $i$. Hence,
$$g\cdot p \in g \cdot(e,g^{n-1}) =(g,g^n) \not\ni p.$$
On the other hand, if there exists $1< i <n$ such that $g^i \notin (p,g)$, let $j$ be the smallest such $i$, so that $p \in (e,g^{j-1})$. Then
$$g \cdot p \in g \cdot (e,g^{j-1}) = (g,g^j) \not\ni p.$$
\item If $p \in (g,h)$, we have two subcases similar to before. If for all $1< i < n$, $g^i \in (g,p)$, then
$$g\cdot p \in g \cdot (g^{n-1},e) = (g^n,g) \not\ni p.$$
On the other hand, if there exists $1< i< n$ such that $g^i \notin (g,p)$, let $j$ be the smallest such $i$, so that $p \in (g^{j-1},e)$. Then
$$g \cdot p \in g \cdot (g^{j-1},e) = (g^j,g) \not\ni p.$$
\item If $p \in (h,g^n)$, then $p \in (g^i,e)$ for all $1\leq i < n$. In particular, $p \in (g^{n-1},e)$, hence
$$g \cdot p \in g \cdot (g^{n-1},e) = (g^n,g) \not\ni p.$$
\item If $p \in (g^n,e)$, then, since $g$ is not torsion, $g^n$ is also not torsion. We can now adapt the argument of cases 1 and 2, with $g$ replaced by $g^n$ to show that $g^n$ does not stabilize $p$. Of course this is enough to show $g$ does not stabilize $p$ either.
\end{enumerate}
If $n=2$, then we can go through these cases again. We leave it to the reader to modify the arguments to show the conclusion holds for each.
\end{proof}

%For the converse, suppose $r_c$ is not free. Let $g \in G$ be an element which has a fixed point. Let $I$ be the connected component of the complement of the set of fixed points of $g$ which contains $p$. If there is no $h \in G$ such that $h(p) \notin I$, then the entire orbit of $p$ is contained in $I$. In particular, this implies that $c$ in fact a linear order which is impossible by Lemma \ref{lem:Archimedeangenuine}. Hence, there exists $h \in G$ such that $h(p) \notin I$. In particular, $c(p, r_c(g^n)(p), r_c(h)(p))$ does not depend on $n$, which means $c$ is not Archimedean. 

The proof of Proposition 3.6 of \cite{Navas10} shows that the dynamical realization of an Archimedean linear order as an action on $\RR$ is free. On the other hand, if one views it as a circular order and considers the dynamical realization as an action on $S^1$, such an action necessarily has a global fixed point. In summary, Archimedean linear orders corresponds to free actions on $\RR$, and Archimedean circular orders correspond to free actions on $S^1$. 

H\"{o}lder showed that (see Theorem 6.10 in \cite{Ghys01} for instance) any group acting freely on either $\RR$ or $S^1$ is abelian.  Therefore linear or circular Archimedean orders exist only for abelian groups both for linear and circular orders.  
%Obviously no Archimedean orders lie in $\COfin$.  

%%%%%%%%%%%%%
%%%%%%%%%%%%%
\section{Free products}\label{sec:freeproducts}
One might try to understand circular orders on 3-manifold groups using the amalgamated product presentations arising from Heegaard splittings.  To this end, we initiate a study of circular orders on free products.  We remark that at this time it is unclear how to deal with amalgamations of free products, since, for example, the Weeks' manifold admits no circular orders \cite{CalegariDunfield03}, but there appears to be no known computable criterion for the existence of orders on 3-manifold groups.

\subsection{Existence}\label{ss:freeexistence}
We show here that a free product of groups $G*H$ is circularly orderable if and only if both $G$ and $H$ are circularly orderable.  Of course, by restriction, one direction of this equivalence is obvious.  To prove the other direction, we need the following well-known lemma. 

\begin{lem}\label{lem:ses}
Let $0 \to K \to G \to H \to 0$ be a short exact sequence where $K$ is $\LO$ and $H$ is $\CO$.  Then $G$ is circularly orderable in such a way that the maps $K \to G$ and $G \to H$ are order preserving. 
\end{lem} 

\begin{proof}
This is rather classical. See, for instance, Lemma 2.2.12 of \cite{Calegari04}.
\end{proof} 

Now we prove the converse: 
\begin{thm}\label{thm:free}
Let $G$ and $H$ be groups with circular orders $c_G$ and $c_H$. Then $G \ast H$ is orderable in a way that extends $c_G$ and $c_H$.
\end{thm}
\begin{proof} Since $G$ and $H$ are circularly orderable, they act faithfully on $S^1$. By the universal property of free products, this defines an action of $G\ast H$ on $S^1$, call it $\rho: G*H \to S^1$. Since $\im(\rho)$ acts faithfully on $S^1$, $\im(\rho)$ is circularly orderable. On the other hand, $\ker(\rho)$ is a normal subgroup of $G \ast H$ that does not intersect $G$ or $H$, since each of $G$ and $H$ acts faithfully. Therefore, by the Kurosh subgroup theorem, $\ker(\rho)$ is a free group. In particular, $\ker(\rho)$ is $\LO$. By Lemma \ref{lem:ses}, $G\ast H$ is $\CO$. 
\end{proof} 

The previous theorem gives an existence result, but is non-constructive. In the next result, we construct an explicit order on $G \ast H$ extending initial orders on $G$ and $H$.

\begin{thm}\label{thm:freeproducts}
Let $G$ and $H$ be circularly ordered groups, with orders $c_G$ and $c_H$.  Then there exists a unique circular order $c$ on $G*H$ that is lexicographical with respect to $c_G$ and $c_H$.  More precisely, there exists a unique circular order $c$ on $G*H$ that satisfies the initial conditions
$$c\mid G^3 = c_G, \ \ c\mid H^3 = c_H, \ \ c(e,g,h) = +1$$
$$c(g_1,g_2,h) = c_G(g_1,g_2,e), \ \text{ and } \ c(g,h_1,h_2) = c_H(e,h_1,h_2)$$
for all $g,g_1,g_2 \in G \setminus \{e\}$, and $h,h_1,h_2 \in H \setminus \{e\}$, together with the lexicographical condition
$$c(xw_1,w_2,w_3) = c(x, w_2,w_3)$$
for all reduced words $xw_1,w_2$ and $w_3$ in $G*H$ such that $x$ is not the leftmost letter of $w_2$ or $w_3$.
\end{thm}

Before proving this, let's consider an example to clarify what we mean by lexicographical order, and to indicate where the cocycle description in the theorem comes from.  In doing so, we will give a sketch of the proof from a topological perspective.  The only essential detail missing is a proof that the final subset $\Gamma_\infty$ of the plane is circularly ordered.

Let $G = \ZZ = \langle a \rangle$ and $H = \ZZ/3\ZZ = \langle b \mid b^3 = e\rangle$.  Let $c_G$ be a rotation order on $G$ with rotation angle $\theta$, and let $c_H$ be the order realized by
$$\begin{aligned} r:\ZZ/3 \ZZ &\to S^1 \\ b^k &\mapsto k/3.\end{aligned}$$
Let $S^1_G$ be a copy of the circle together with an orientation-preserving embedding $G \into S^1_G$, which we think of as a marking of some points of $S^1_G$ by elements of $G$.  Similarly, let $S^1_H$ be a copy of $S^1$, together with the three points $0,1/3$ and $2/3$ marked by the corresponding elements of $\ZZ/3\ZZ$.  Wedge the two circles together at 0 to form a planar graph
$$\Gamma := S^1_G \vee S^1_H$$
as in Figure \ref{fig:preseed}.  We shall blow-up $\Gamma$ to a seed $\Gamma_0$ which generates a new planar graph $\Gamma_\infty$ on which $G*H$ will act faithfully.

\begin{figure}[h]
\begin{center}
\includegraphics{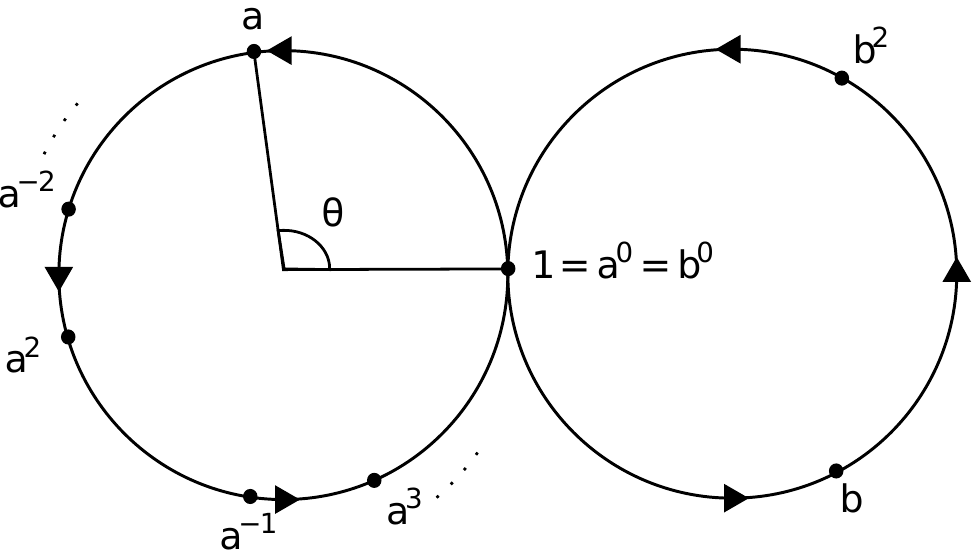}
\caption{The marked wedge $S^1_G \vee S^1_H$.}
\label{fig:preseed}
\end{center}
\end{figure}

First modify $\Gamma$ by blowing up every point in $G\cup H \subset \Gamma$ to an interval, and marking one of the endpoints of the resulting interval with the same marking as before.  Be sure to pick the same endpoint for all intervals (with respect to the orientation of $\Gamma$), as in Figure \ref{fig:seed}.  Call this graph $\Gamma_0$.

\begin{figure}[h]
\begin{center}
\includegraphics{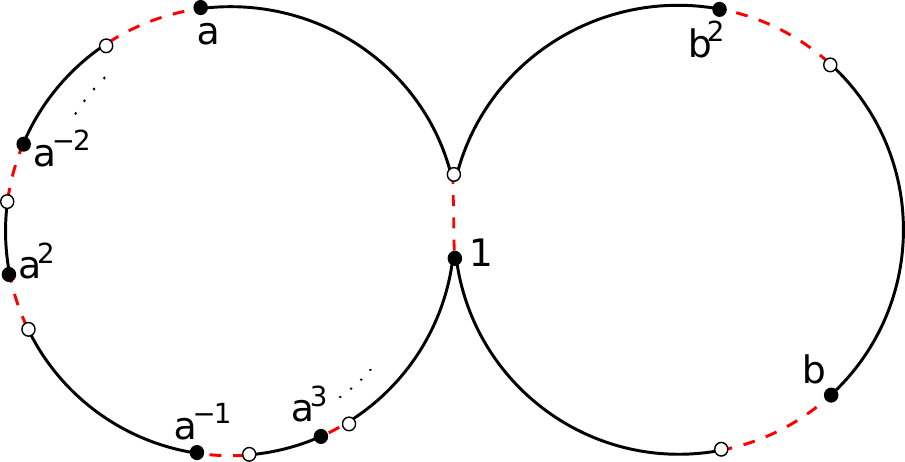}
\caption{The seed $\Gamma_0$.}
\label{fig:seed}
\end{center}
\end{figure}

Now let $w \in G*H$ be any word of length one.  That is, $w$ is an element of either $G$ or $H$.  If $w$ is in $H$, take a copy of $S^1_G$ together with the marking by $G$, and relabel every marked point by appending $w$ onto the left end.  Call this marked circle $S^1_{w\cdot G}$.  Blow it up along the marked points and glue the result onto $\Gamma_0$ along the edge that contains the marking $w$, so that this piece is contained in the closure of the unbounded component of $\RR^2 \setminus \Gamma_0$.  Do this for all the elements $w$ in $G$ and $H$.  Of course, this requires a choice of how to squeeze the infinitely many circles into the plane, but these choices will not affect the final order we construct.  Call the resulting planar graph $\Gamma_1$.  Supposing inductively that $\Gamma_l$ is construction, we can repeat a similar procedure to construct $\Gamma_{l+1}$.  See Figure \ref{fig:intermediate}.

\begin{figure}[h]
\begin{center}
\includegraphics{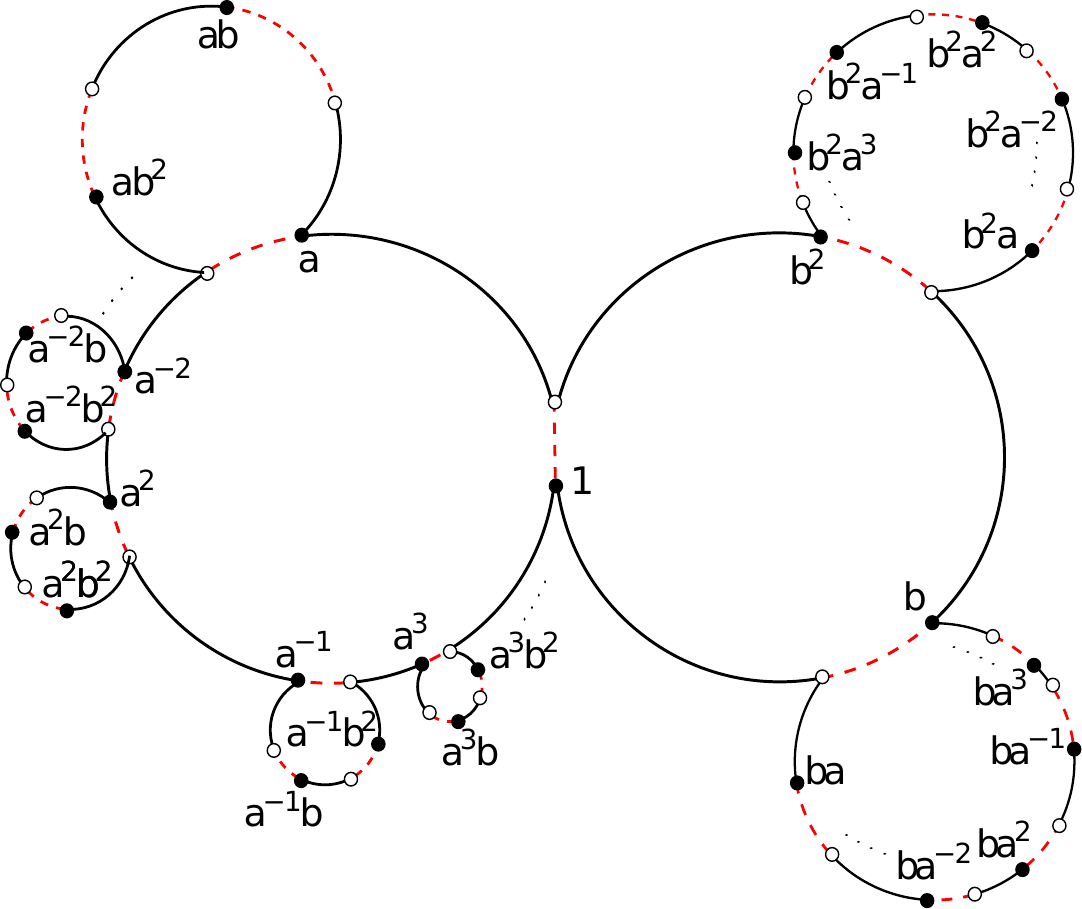}
\end{center}
\caption{$\Gamma_1$, an intermediate step in the construction of $\Gamma_\infty$.}
\label{fig:intermediate}
\end{figure}

Notice that $\Gamma_l \subset \Gamma_{l+1}$ for all $l$.  Define
$$\Gamma_\infty' := \bigcup_{l \geq 0} \Gamma_l.$$
$G*H$ acts on $\Gamma_\infty'$ in a way that continuously extends the actions of $G$ and $H$ on the blown-up copies of $S_G^1$ and $S_H^1$ inside $\Gamma_\infty'$, respectively.  Let $E$ be the set of points in the interiors of the internal edges of $\Gamma_\infty'$.  That is, $E$ consists of the orbit of the interiors of the edges constructed when blowing up $\Gamma$ to $\Gamma_0$.  Define
$$\Gamma_\infty := \Gamma_\infty' \setminus E.$$
$\Gamma_\infty$ contains the orbit of the point marked e, which moreover is trivially stabilized.  Then the order structure $c$ induced by the free action of $G*H$ on e extends $c_G$ and $c_H$.

To connect this to Theorem \ref{thm:freeproducts}, observe that the marking of the seed $\Gamma_0$ contains the data of the initial conditions mentioned in the theorem.  The lexicographical condition can be understood by observing that the order of a triple of long words in $G*H$ alternating between $G$ and $H$ depends only on their leftmost letter, so long as the leftmost letters of the three words are distinct; this is clearly the case with the dynamical construction we just outlined.

One might try to adapt this example to prove Theorem \ref{thm:freeproducts}.  The main difficulty lies in showing the final graph in the construction is circularly orderable as a set, which would amount to checking a cocycle condition anyway.

\begin{proof}[Proof of Theorem \ref{thm:freeproducts}]
Using the diamond lemma, we first prove uniqueness, after which it will be easy to define $c$.  Consider the following three types of reductions on the set $(G*H)^3$ of ordered triples of reduced words in $G * H$:
\begin{enumerate}
\item If $x \in G \cup H$ and $w_1 = xw_1', w_2=xw_2'$ and $w_3 = xw_3'$ are reduced,
$$(w_1,w_2,w_3) \mapsto (w_1',w_2',w_3').$$
\item If $x \in G \cup H$ is the leftmost letter of precisely two words, left multiply the triple by $x^{-1}$ and reduce.  For example, if $w_1 = xw_1'$, $w_3 = xw_3'$ and $w_2$ does not start with $x$, then
$$(w_1,w_2,w_3) \mapsto (w_1',x^{-1}w_2,w_3'),$$
where it is understood that $x^{-1}w_2$ has been reduced if need be.  There are two more versions of this move, which we do not make explicit.
\item If $x \in G \cup H$ is the leftmost letter of precisely one word, remove the subword to the right of $x$.  For example, if $w_2 = xw_2'$ and neither $w_1$ nor $w_3$ begins with $x$, then
$$(w_1,w_2,w_3) \mapsto (w_1,x,w_3).$$
\end{enumerate}
All three reductions strictly decrease the sum of the three word lengths involved.  Thus there are no infinite sequences of reductions.  It is completely straightforward to check cases and show that two reductions of a fixed triple have a mutual reduction. Indeed, if a triple can have a type 1 reduction applied, it is unique, and that triple can not have a type 2 or type 3 applied. If a triple can have a type 2 reduction applied, then it can always also have a unique type 3 reduction applied. For instance, let $w_1 = xw_1', w_2 = yw_2', w_3 = xw_3'$ be reduced words. Then the following digram shows the two different sequences of reductions to reach a mutual reduction. 
\[
\begin{tikzcd}
{} & (xw_1', yw_2', x w_3') \arrow[swap]{dl}{2.} \arrow{dr}{3.} & {}\\
(w_1', x^{-1}yw_2', w_3') \arrow[swap]{dr}{3.} & {} & (xw_1', y, xw_3')\arrow{dl}{2.}  \\
{} & (w_1', x^{-1}y, w_3') & {}
\end{tikzcd}
\]
The final case to consider is a triple that admits three different type 3 reductions. More precisely, let $w_1 = xw_1', w_2 = yw_2', w_3 = zw_3'$, where $x,y$ and $z$ are all distinct. Then clearly $(x,y,z)$ is a mutual reduction for all three of the type 3 reductions of $(w_1,w_2,w_3)$. We conclude by the diamond lemma (also known as Newman's lemma) that every triple can be reduced to a unique minimal triple, \ie a triple which admits no reductions.  Moreover, it is easy to verify that the minimal triples are precisely the elements of $(G\cup H)^3$.

Any circular order $c$ satisfying the lexicographical condition in the statement of the theorem must be invariant under the 3 reductions.  Indeed, Reductions 1 and 2 follow from left invariance, and Reduction 3 follows from the cocycle condition together with the lexicographical condition.  Since the minimal reduction of every triple is in $(G \cup H)^3$, $c$ is uniquely specified by imposing the initial conditions.  So we define the unique lexicographical extension of $c_G$ and $c_H$ that satisfies the initial conditions of the theorem by
$$c(w_1,w_2,w_3) := c(x,y,z),$$
where $(x,y,z) \in (G \cup H)^3$ is the minimal reduction of $(w_1,w_2,w_3)$.

Note that homogeneity of $c$ follows immediately from the definition of the minimal reduction of a triple. We conclude by using induction to show that $c$ is a cocycle.  Our induction occurs inside $(G*H)^4$, where $(w_1,w_2,w_3,w_4) \geq (w_1',w_2',w_3',w_4')$ if each $w_i'$ is a subword of $w_i$. It is straightforward to show this is a Noetherian poset, so our induction is well-founded.

Before beginning, we observe that since the minimal reduction of a triple is equivariant with respect to the action of the symmetric group $S_3$ on triples, $c$ is invariant under cyclic permutations of its input and antisymmetric with respect to transposing two of its arguments. This uses the initial conditions and the fact that $c_G$ and $c_H$ both have both of these properties. We will need these properties in what follows.

For the base cases, consider a quadruple $(w,x,y,z) \in (G\cup H)^4$.  We need to show
$$dc(w,x,y,z) = 0,$$
where
$$dc(w,x,y,z) = c(x,y,z) - c(w,y,z) + c(w,x,z) - c(w,x,y).$$
If all four elements are in $G$ or all four are in $H$, $dc(w,x,y,z) = 0$ because $c_G$ and $c_H$ are cocycles.  Suppose, by way of example, that $(w,x,y,z) = (g_1,h_1,h_2,g_2)$ where $g_1,g_2 \in G \setminus \{e\}$ and $h_1,h_2 \in H \setminus \{e\}$.  Then
$$\begin{aligned} dc(g_1,h_1,h_2,g_2) &= c(h_1,h_2,g_2) - c(g_1,h_2,g_2) + c(g_1,h_1,g_2) - c(g_1,h_1,h_2)\\
&=c_H(h_1,h_2,e) - c_G(g_1,e,g_2) + c_G(g_1,e,g_2) - c_H(e,h_1,h_2) \\&= 0.\end{aligned}$$
All the cases---namely, $(g_1,g_2,g_3,h_1)$, $(g_1,g_2,h_1,h_2)$, $(g_1,h_1,h_2,h_3)$ and their permutations---are similar. Indeed, the initial conditions immediately imply
$$dc(g_1,g_2,g_3,h_1) = dc(g_1,g_2,g_3,e)$$
and
$$dc(g_1,h_1,h_2,h_3) = dc(e,h_1,h_2,h_3),$$
so these cases (and similarly, their permutations) easily follow because $dc|G = 0$ and $dc|H = 0$, respectively. The remaining three cases are the permutations of $(g_1,g_2,h_1,h_2)$ (modulo symmetry of the cases with respect to reindexing and switching $G$ and $H$), one of which we showed above, and the other two of which are here:
$$\begin{aligned} dc(g_1,g_2,h_1,h_2) &= c(g_2,h_1,h_2) - c(g_1,h_1,h_2) + c(g_1,g_2,h_2) - c(g_1,g_2,h_1)\\
&=c(e,h_1,h_2) - c(e,h_1,h_2) + c(g_1,g_2,e) - c(g_1,g_2,e) = 0\end{aligned}$$
and
$$\begin{aligned} dc(g_1,h_1,g_2,h_2) &= c(h_1,g_2,h_2) - c(g_1,g_2,h_2) + c(g_1,h_1,h_2) - c(g_1,h_1,g_2)\\
&= c(h_1,e,h_2) - c(g_1,g_2,e) + c(e,h_1,h_2) - c(g_1,e,g_2) = 0.\end{aligned}$$

Consider a quadruple of reduced words, which, without loss of generality, we suppose is of the form $(xw_1,w_2,w_3,w_4)$ where $xw_1$ is the longest word in the quadruple. (To see why this is acceptable, note that if $c$ is symmetric with respect to cyclic permutations, and antisymmetric with respect to transpositions, then
$$dc(w,x,y,z) = - dc(x,y,z,w),$$
so that one side is 0 if and only if the other side is too.) In particular, we assume $w_1$ is not the empty word, since that would put us back in the base case. We suppose inductively that for every $(v_1,v_2,v_3,v_3) \leq (xw_1,w_2,w_3,w_4)$, $dc(v_1,v_2,v_3,v_4) = 0$. By definition
$$dc(xw_1,w_2,w_3,w_4) = c(w_2,w_3,w_4) - c(xw_1,w_3,w_4) + c(xw_1,w_2,w_4) - c(xw_1,w_2,w_3).$$
To compute any further, we consider several cases, based on the combinatorics of the reduced words:
\begin{enumerate}
\item $x$ does not begin $w_2,w_3$ or $w_4$
$$\begin{aligned} dc(xw_1,w_2,w_3,w_4) &= c(w_2,w_3,w_4) - c(xw_1,w_3,w_4) + c(xw_1,w_2,w_4) - c(xw_1,w_2,w_3) \\ &= c(w_2,w_3,w_4) - c(x,w_3,w_4) + c(x,w_2,w_4) - c(x,w_2,w_3) \\
&= dc(x,w_2,w_3,w_4) = 0,\end{aligned}$$
where the last equality follows by induction, and the assumption that $w_1$ is not the empty word. The next three cases will not need induction.
\item $w_2 = xw_2'$, $x$ does not begin $w_3$ or $w_4$
$$\begin{aligned} dc(xw_1,w_2,w_3,w_4) &= dc(xw_1,xw_2',w_3,w_4) \\
&=c(xw_2',w_3,w_4) - c(xw_1,w_3,w_4) + c(xw_1,xw_2',w_4) - c(xw_1,xw_2',w_3)\\
&=c(x,w_3,w_4) - c(x,w_3,w_4) + c(w_1,w_2',x^{-1}w_4) - c(w_1,w_2',x^{-1}w_3)\\
&= 0 + c(w_1,w_2',x^{-1}) - c(w_1,w_2',x^{-1}) =0. \end{aligned}$$
\item $w_3 = xw_3'$, $x$ does not begin $w_2$ or $w_4$
$$\begin{aligned} dc(xw_1,w_2,w_3,w_4) &= dc(xw_1,w_2,xw_3',w_4) \\
&=c(w_2,xw_3',w_4) - c(xw_1,xw_3',w_4) + c(xw_1,w_2,w_4) - c(xw_1,w_2,xw_3')\\
&=c(w_2,x,w_4) - c(w_1,w_3',x^{-1}w_4) + c(x,w_2,w_4) - c(w_1,x^{-1}w_2,w_3')\\
&=c(w_2,x,w_4) + c(x,w_2,w_4) - c(w_1,w_3',x^{-1}) - c(w_1,x^{-1},w_3')\\
&= 0 - 0 = 0. \end{aligned}$$
\item $w_4 = xw_4'$, $x$ does not begin $w_2$ or $w_3$
$$\begin{aligned} dc(xw_1,w_2,w_3,w_4) &= dc(xw_1,w_2,w_3,xw_4')\\
&=c(w_2,w_3,xw_4') - c(xw_1,w_3,xw_4') + c(xw_1,w_2,xw_4') - c(xw_1,w_2,w_3)\\
&=c(w_2,w_3,x) - c(w_1,x^{-1}w_3,w_4') + c(w_1,x^{-1}w_2,w_4') - c(x,w_2,w_3)\\
&=c(w_2,w_3,x) - c(x,w_2,w_3) - c(w_1,x^{-1},w_4') + c(w_1,x^{-1},w_4')\\
&= 0.\end{aligned}$$
\item $w_2 = xw_2'$, $w_3=xw_3'$ and $x$ does not begin $w_4$. Write $w_4 = yw_4'$. Then
$$\begin{aligned} dc(xw_1,w_2,w_3,w_4) &= dc(xw_1,xw_2',xw_3',yw_4')\\
&=c(xw_2',xw_3',yw_4') - c(xw_1,xw_3',yw_4') + c(xw_1,xw_2',yw_4') - c(xw_1,xw_2',xw_3')\\
&=c(w_2',w_3',(x^{-1}y)w_4') - c(w_1,w_3',(x^{-1}y)w_4') + c(w_1,w_2',(x^{-1}y)w_4') - c(w_1,w_2',w_3')\\
&=c(w_2',w_3',(x^{-1}y)) - c(w_1,w_3',(x^{-1}y)) + c(w_1,w_2',(x^{-1}y)) - c(w_1,w_2',w_3')\\
&=c(xw_2',xw_3',y) - c(xw_1,xw_3',y) + c(xw_1,xw_2',y) - c(xw_1,xw_2',xw_3')\\
&=dc(xw_1,xw_2',xw_3',y)=0,\end{aligned}$$
where the fourth equality follows from reduction 1 (since $(x^{-1}y)$ can't begin $w_1,w_2'$ or $w_3'$), and the last equality follows by induction. Note that if $w_4'$ is the empty word, this calculation has not shown anything, so we have some subcases. Write $xw_1 = vav_1, w_2=xw_2'=vbv_2, w_3=xw_3'=vcv_3, w_4=y$, where $v$ is the longest rightmost common subword between $xw_1,w_2$ and $w_3$, and $a,b,c\in G\cup H$. In particular, $a,b$ and $c$ are not all the same.

Suppose $a,b$ and $c$ are all distinct, and let $d \in G\cup H$ be the first letter of $v^{-1}$. Then by homogeneity and type 3 reductions
$$\begin{aligned} dc(xw_1,w_2,w_3,w_4) &= dc(av_1,bv_2,cv_3,v^{-1}y) \\
&= c(bv_2,cv_3,v^{-1}y) - c(av_1,cv_3,v^{-1}y) + c(av_1,bv_2,v^{-1}y) - c(av_1,bv_2,cv_3) \\
&= c(b,c,d) - c(a,c,d) + c(a,b,d) - c(a,b,c) \\
&= dc(a,b,c,d)=0,\end{aligned}$$
where the last equality follows from the base case.

If precisely two of $a,b$ and $c$ are equal, then we proceed as in Cases (2), (3) or (4) above.
\item $w_2 = xw_2'$, $w_4=xw_4'$ and $x$ does not begin $w_3$. Follows like (5).
\item $w_3 = xw_3'$, $w_4=xw_4'$ and $x$ does not begin $w_2$. Follows like (5).
\item $w_2=xw_2', w_3 = xw_3', w_4 = xw_4'$.
$$\begin{aligned} dc(xw_1,w_2,w_3,w_4) &= dc(xw_1,xw_2',xw_3',xw_4') \\ &= dc(w_1,w_2',w_3',w_4') = 0,\end{aligned}$$
by homogeneity and induction.
\end{enumerate}
This completes the proof that $c$ is a cocycle.
\end{proof}

%Suppose inductively that $dc(w_1,w_2,w_3,w_4) = 0$.   Let  We argue again by example, leaving the other cases for the reader to verify.  Let $x \in G \cup H\setminus \{e\}$ and suppose $xw_1$ is reduced.  Furthermore, suppose $x$ does not begin $w_2$ or $w_4$, but $w_3 = xw_3'$.  Then
%$$\begin{aligned} dc(xw_1,w_2,w_3,w_4) &= c(w_2,w_3,w_4) - c(xw_1,w_3,w_4) + c(xw_1,w_2,w_4) - c(xw_1,w_2,w_3) \\
%&= c(w_2,xw'_3,w_4) - c(xw_1,xw'_3,w_4) + c(xw_1,w_2,w_4) - c(xw_1,w_2,xw'_3)\\
%&= c(w_2,x,w_4) - c(w_1,w'_3,x^{-1}w_4) + c(x,w_2,w_4) - c(w_1,x^{-1}w_2,w'_3)\\
%&= - c(w_1,w'_3,x^{-1}) - c(w_1,x^{-1},w'_3)\\
%&= 0,\end{aligned}$$
%where the cancellations
%$$c(w_2,x,w_4) + c(x,w_2,w_4) = 0$$
%and
%$$c(w_1,w'_3,x^{-1}) + c(w_1,x^{-1},w'_3) = 0,$$
%occur because, as observed above, $c$ is antisymmetric with respect to transpositions.

Theorem \ref{thm:freeproducts} says there is a unique way to extend circular orders on two groups to a circular order $c$ on their free product satisfying certain initial conditions.  The initial conditions could also be described as an ``interleaving pattern."  It could be interesting to understand, for two fixed starting orders, which interleaving patterns are admissible for constructing a (unique) extending circular order.  Loosely speaking, the proof of the next Theorem \ref{thm:COminperfect} exploits perturbations of interleaving patterns to show a certain subset of $\CO(G*H)$ has no isolated points.  We shall use very special perturbations though, for which we can guarantee that the stabilizer of the marked point with respect to the perturbed action is free (and hence linearly orderable).  For arbitrary perturbations, it is unclear how to understand the resulting stabilizer.

%%%%%%%%%%%%%
%%%%%%%%%%%%%
\subsection{Abundance}\label{ss:freeabundance}
Rivas showed that the space of linear orders of a free group $\LO(F_n)$ does not have an isolated point, and, hence, is a Cantor set \cite{Rivas12}.  In this subsection, we give a partial generalization of the result to circular orders. 

Before we proceed to our main result of this section, we note that Lemma \ref{lem:ses} admits a dynamical interpretation, which we formalize below. 

Let $X$ be a $G$-set, i.e., a set which admits a left $G$-action. A circular order on $X$ is said to be $G$-invariant if for each $g \in G$, $X \xrightarrow{g} X, x \mapsto gx$, is order-preserving. More precisely, for all $g \in G$, $x_1, x_2, x_3 \in X$, one has $c(x_1, x_2, x_3) = c(gx_1, gx_2, gx_3)$. The following lemma slightly generalizes Lemma \ref{lem:ses}, and since we found no literature stating this in this generality, we include the proof. 

\begin{lem} 
\label{lem:CO_Gset}
Let $X$ be a $G$-set with a $G$-invariant circular order. Suppose $\Stab_G(x_0)$ is $\LO$ for some $x_0 \in X$. Then $G$ is $\CO$ so that the inclusion map $\Stab_G(x_0) \to G$ and the map $\phi: G \to X, g \mapsto g x_0$, respect the orders in an appropriate sense.  
\end{lem} 
\begin{proof} 
First we pick a left-invariant linear order $\leq$ on $\Stab_G(x_0)$. We define a linear order on each coset $g \Stab_G(x_0)$ as follows. For any two elements $h, h'$ of $g \Stab_G(x_0)$, we say $h < h'$ if and only if $g^{-1}h < g^{-1}h'$ with respect to the linear order we chose on $\Stab_G(x_0)$. To see this is well-defined, assume $g \Stab_G(x_0) = g' \Stab_G(x_0)$. Then $g'^{-1}g \in \Stab_G(x_0)$. By the left-invariance, $g^{-1}h < g^{-1}h'$ if and only if $g'^{-1}g g^{-1}h < g'^{-1}g g^{-1}h'$, i.e., $g'^{-1}h < g'^{-1}h'$. Hence, our linear on a coset does not depend on the choice of a representative. 

Let's finally define a circular order $c$ on $G$. The recipe is almost exactly same as the one given in the proof of Lemma 2.2.12 of \cite{Calegari04}. 
%Define $\phi: G \to K$ by $\phi(g) = g \Stab_G(x_0)$. 
For each distinct triple $g_1, g_2, g_3$ of elements of $G$, we circularly order them as follows: 
\begin{itemize}
\item[(1)] If $\phi(g_1), \phi(g_2), \phi(g_3)$ are distinct, circularly order them by the circular order on their image in $X$.
\item[(2)] If $\phi(g_1) = \phi(g_2)$ but these are distinct from $\phi(g_3)$, then $g_1$ and $g_2$ belong to the same coset. If $g_1 < g_2$ with respect to the linear order we defined on the coset then $g_1, g_2, g_3$ is positively ordered, and otherwise it is negatively ordered. 
%$g_2^{-1}g_1 \in K$. If $g_2^{-1}g_1 < \Id$ then $g_1, g_2, g_3$ is positively ordered, otherwise it is negatively ordered. 
\item[(3)] If $\phi(g_1) = \phi(g_2) = \phi(g_3)$ then $g_1, g_2, g_3$ are all in the same coset. If $g_1 < g_2 < g_3$ (up to cyclic permutation) then they are negatively ordered. 
%$g_3^{-1}g_1, g_3^{-1}g_2, \Id$ are all in $K$, and therefore inherit a total order. The three corresponding elements of G in the same total order are negatively ordered.
\end{itemize} 
One can easily check that this defines a left-invariant circular order on $G$. 
\end{proof} 

Note that, in the above lemma, one can easily see that the set, say $K$, of all cosets of $\Stab_G(x_0)$ has a natural circular order so that $K$ is order-isomorphic to the orbit of $x_0$ under $G$. More precisely, we can define a circular order on the set $K$ of cosets of $\Stab_G(x_0)$ by using the circular order on $X$. We say $(g_1 \Stab_G(x_0), g_2 \Stab_G(x_0), g_3 \Stab_G(x_0))$ is positively oriented if and only if $(g_1x_0, g_2 x_0, g_3 x_0)$ is positively oriented with respect to the $G$-invariant circular order on $X$. Again, this is well-defined, since two representatives of each coset differ by an element of $\Stab_G(x_0)$ so the image of $x_0$ under two different representatives coincide. Hence, the lemma can be seen as a generalization of Lemma 2.2.12 of \cite{Calegari04}, in the sense that one can construct a circular order on a group $G$ from a linear order on a (not necessarily normal)  subgroup $H$, and a circular order on the set of cosets which is compatible with the left-action of $G$. This formulation is useful especially when one wants to circularly order a group from its action on the circle which may not have a trivially stabilized point, or even not be faithful. 

\begin{prop}
\label{prop:nonfaithfulaction}
Let $G$ be a group acting (not necessarily faithfully) on $S^1$, and $p$ a point in $S^1$. If $\Stab_G(p)$ is $\LO$, then $G$ is $\CO$. 
\end{prop} 
\begin{proof} The orbit of $p$ under the action of $G$ is a $G$-set with a circular order inherited from the circle. Obviously this order is $G$-invariant. The claim follows from an application of Lemma \ref{lem:CO_Gset} by setting $X$ to be the orbit of $p$ under the action of $G$, and $x_0 = p$. 
\end{proof}

For a set $A$ of elements of a group $G$, two circular orders $c$ and $c'$ are said to coincide on $A$ if they coincide on the set of triples consisting of elements in $A$. 

\begin{thm}
\label{thm:COminperfect}
Let $G, H$ be countable infinite groups. %such that $G \ast H$ is $\LO$. 
Suppose a circular order $c$ of $G \ast H$ admits a dynamical realization $r_c$ which restricts to a minimal action of $G$, \ie all orbits under the action of $G$ are dense in $S^1$. Then $c$ is not an isolated point of $\CO(G \ast H)$.
\end{thm} 

\begin{proof}
It is enough to show that for an arbitrary finite set $S$ of elements of $G\ast H$, there exists a circular order $c'$ on $G \ast H$ which coincides with $c$ on $S$ but such that $c' \neq c$.
For each $s \in S$, consider it as an alternating product of elements of G and elements of H, and add all rightmost sub-words of $s$ to $S$. For instance, if $s$ is $g_1h_1g_2h_2$, then we add $h_2, g_2h_2, h_1g_2h_2$ to $S$. We also add the group identity element e to $S$. If a circular $c'$ coincides with $c$ on this enlarged set $S$, then obviously they coincide on the original finite set of elements of $G \ast H$. 

 We shall construct $c'$ by conjugating the action of one of the factors in the free product by a homeomorphism of $S^1$ supported on a small interval, thereby modifying the interleaving pattern of $G$ and $H$. 
 
Let $r_c : G \ast H \to \Homeop(S^1)$ be the dynamical realization of $c$ with marked point $p \in S^1$. As usual, we consider $G \ast H$ as a subset of $S^1$ using $r_c$, \ie as the orbit of $p$.  Let $I \subset S^1$ be a connected component of the complement of the elements of $G\ast H$ involved in $S$, such that $I$ contains a point $g$ in the $G$-orbit of $p$ and a point $h$ in the $H$-orbit of $p$.  $I$, $g$ and $h$ exist because $H$ is infinite and the $G$ action is minimal by our assumption.  Choose an open connected arc $J$ which is properly contained in $I$ and contains both $g$ and $h$. Let $r_c^G : G \to \Homeop(S^1)$ be the restriction of $r_c$ to the factor $G$, and define $r_c^H$ similarly.
 
Let $\alpha$ be a homeomorphism of the circle satisfying the following properties:
\begin{itemize}
\item[(i)] $\alpha$ is the identity map on $S^1 \setminus J$.
\item[(ii)] the circular order of the triple $(e, g, h)$ is different from the circular order of $(e, g, \alpha(h))$, where $e$ is one of the endpoints of $J$, and
\item[(iii)] $\alpha$ induces a bijection on $G\ast H$ (considered as a subset of $S^1$ according to the unperturbed order). 
% the images of $H \subset S^1$ and $G \subset S^1$ under $\alpha$ and $\alpha^{-1}$ are disjoint.
\end{itemize}
Such $\alpha$ exists since any countable dense subset of an open interval is order-isomorphic to $\QQ$ with the usual order relation. 
Now perturb $r_c^H$ by conjugating with $\alpha$.  Namely, define $r_\alpha^H = \alpha r_c^H \alpha^{-1}$. Let $r_\alpha : G\ast H \to \Homeop(S^1)$ be the canonical homomorphism determined by $r_c^G$ and $r_\alpha^H$.  Note that $r_\alpha$ may not be injective. In particular, $K := \Stab_{r_\alpha(G \ast H)}(p)$ may not be trivial. We show that $K$ is a free group, hence left-orderable. 

By Kurosh's subgroup theorem, it is enough to show that $K$ trivially intersects the conjugates of G and H in $r_\alpha(G \ast H)$. %This is equivalent to the following: 
Suppose the negation. For instance, say $K$ intersects $r_\alpha(wHw^{-1})$ for some $w \in G \ast H$. That means, there exists $h$ such that $(r_\alpha(w))(\alpha h \alpha^{-1}) (r_\alpha(w)^{-1})$ fixes $p$ where $h, w$ are their images under the original action $r_c$. This is equivalent to that $h$ fixes $\alpha^{-1} r_\alpha(w)^{-1} (p)$. 

Suppose $w$ is written as $g_1h_1g_2h_2 \ldots g_nh_n$ as a reduced word. \\ Then $r_\alpha(w) = 
g_1 \alpha h_1 \alpha^{-1} g_2 \alpha h_2 \alpha^{-1} \ldots g_n \alpha h_n \alpha^{-1}$. Understanding $g_i$ and $h_i$ are their images under $r_c$, now we have that $h$ fixes 
$h_n^{-1} \alpha^{-1} g_n^{-1} \ldots g_1^{-1} (p)$. But since $\alpha$ maps the $r_c(G\ast H)$-orbit of $p$ to itself, this point $h_n^{-1} \alpha^{-1} g_n^{-1} \ldots g_1^{-1} (p)$ must be still in the the $r_c(G\ast H)$-orbit of $p$. But this is impossible, since $h$ does not fix any point in the orbit of $p$ in the original action (recall that $p$ is trivially stabilized by $r_c(G \ast H)$. This implies that $K$ does not interest any conjugate of $H$ in $r_\alpha(G \ast H)$. The exact same argument works if one replaces $H$ by $G$. Hence, $K$ must be free by Kurosh's theorem. Pick an arbitrary left-order of $K$. 

%($\star$) $\space \space$ For every $w \in r_c(G \ast H)$, 
%$\Stab_{\Homeop(S^1)}[(\alpha^{-1}\circ w)(p)] \cap (r_c(G \cup H)) = \emptyset$.  

%But $(\star)$ follows from the property (iii) of $\alpha$ because $p$ is trivially stabilized by $G\ast H$ in the original action.

%Furthermore, the set of cosets of $K$ in $G \ast H$, call it $X$, has a natural circular order inherited from the orbit of $p$ under the action of $r_\alpha(G\ast H)$. $G\ast H$ has a left-action on $X$ which preserves this order, ie., $X$ admits a left-invariant circular order as a $G\ast H$-set. This is enough to construct a circular order on $G \ast H$ so that the natural maps $K \to G\ast H$ and $G \ast H \to X$ are compatible with the given orders. 

Now applying Proposition \ref{prop:nonfaithfulaction} (or rather the proof of Lemma \ref{lem:CO_Gset}) by setting $X$ to be the orbit of $p$ under $r_\alpha(G\ast H)$, and $x_0 = p$, we obtain a circular order on $G \ast H$ from the perturbed action $r_\alpha$, call this new circular order $c'$. 

The circular order of the orbit of $p$ under $r_\alpha$ may be different from one for $r_c$, but the enlarging process for $S$ given at the beginning of the proof ensures that at least the circular order on the set $S$ has not been changed. More precisely, the sets $\{ r_\alpha(s)(p) : s \in S \}$ and $\{ r_c(s)(p) : s \in S \}$ are order-isomorphic. Therefore, $c$ coincides with $c'$ on $S$. But the property (ii) of $\alpha$ ensures that $c \neq c'$, which completes the proof. 
\end{proof} 

One may wonder if the assumption in Theorem \ref{thm:COminperfect} is vacuous. In fact, it is easy to produce an example in this situation. For instance, let $r$ be a rigid rotation by an irrational angle. Then a minor variation of the proof of Proposition 4.5 of \cite{Ghys01} shows that for a generic choice of a loxodromic isometry $f$ of $H^2$, the group generated by $r$ and $f$ is free. The set $X_w$ in Ghys' proof can be replaced by the set $X_w = \{ k \in \RR :  w(r,f_k) \}$ where $w$ is a non-trivial word in $F_2$ and $f_k$ denotes the map $z \to k z$ in the upper half plane. One can easily show that $X_w$ is a proper subset of $\RR$, and then a similar argument goes through to show that for a generic choice of $k$, the group generated by $r$ and $f_k$ is free. In fact, by choosing a finite number of such loxodromic isometries which do not share fixed points and possibly raising their powers, one can produce non-isolated circular orderings of free groups of finite rank.  

It is not clear if one could generalize Theorem \ref{thm:COminperfect} to all of $\CO(G\ast H)$.  Indeed, the existence of $I, g$ and $h$ are needed for the local perturbation argument, and these need not exist without the assumptions of minimality of the $G$ action\footnote{Minimality of $G$ could be relaxed to the condition that the $G$-orbit of the marked point be dense.} and infiniteness of $H$.  For example, in contrast to Rivas' case of linear orders, $G$ and $H$ could be finite cyclic groups and $S$ could involve all the elements of $G$ and $H$.  In this case any perturbation would have to occur on triples involving words of $G*H$ that alternate many times between $G$ and $H$, in which case the universal property of free products can not be so easily exploited.  Nevertheless, we propose the following

\begin{conj}
If $G$ and $H$ are circularly orderable groups, then $\CO(G*H)$ either has no isolated points or is finite.  In particular, $\CO(G*H)$ is a Cantor set if it is infinite.
\end{conj}

%\begin{cor} 
%Let $G, H$ be countable, infinite groups, and $c$ a circular order on $G \ast H$. Then $c$ is not an isolated point in $\CO(G \ast H)$ if one of the followings holds. 
%\begin{itemize} 
%\item[(1)] $c$ is linear, 
%\item[(2)] $c$ restricts as a minimal order to a proper free factor of $G \ast H$, %and $G \ast H$ is left-orderable,  
%\item[(3)] $c$ is Archimedean as a circular order. 
%\end{itemize} 
%\end{cor}
%\begin{proof}
%(1) is a well-known fact for $\LO(G \ast H)$ (see, for instance, \cite{Navas10}). (2) follows from Theorem \ref{thm:COminperfect} and (3) follows from Theorem \ref{thm:Archimedean}. 
%\end{proof} 

%%%%%%%%%%%%%
%%%%%%%%%%%%%
\bibliographystyle{alpha}
\bibliography{CObib}

\end{document}